\documentclass[12pt,a4paper]{article}

\usepackage[latin1]{inputenc}
\usepackage[english]{babel}
\usepackage[usenames, dvipsnames]{xcolor}
\usepackage{amssymb}
\usepackage{amsfonts}
\usepackage{amsmath}
\usepackage{amsthm}
\usepackage{epsfig}
\usepackage{graphics}
\usepackage{dsfont}

\usepackage{verbatim}

\usepackage{latexsym}

\newtheorem{theorem}{Theorem}[section]
\newtheorem{lemma}[theorem]{Lemma}
\newtheorem{remark}[theorem]{Remark}
\newtheorem{remarks}[theorem]{Remarks}
\newtheorem{example}[theorem]{Example}
\newtheorem{proposition}[theorem]{Proposition}
\newtheorem{definition}[theorem]{Definition}
\newtheorem{corollary}[theorem]{Corollary}

\allowdisplaybreaks[4]

\newcommand{\NN}{\mathbb{N}}

\newcommand{\RR}{\mathbb{R}}

\newcommand{\EE}{\mathbb{E}}
\newcommand{\LL}{\mathbb{L}}

\newcommand{\cP}{\mathcal{P}}

\newcommand{\cL}{\mathcal{L}}

\newcommand{\supp}{\mathrm{supp}\,}

\newcommand{\ID}{\mathrm{ID}}
\newcommand{\SD}{\mathrm{L}}
\newcommand{\BO}{\mathrm{BO}}
\newcommand{\T}{\mathrm{T}}

\begin{document}
\title{Exponential functionals of L\'evy processes with jumps}
\author{Anita Behme\thanks{Technische Universit\"at M\"unchen, Zentrum Mathematik, 
Boltzmannstra\ss e 3, D-85748 Garching bei M\"unchen,
Germany; email: a.behme@tum.de, tel.: +49/89/28917424,
fax:+49/89/28917435}}
\date{\today}
\maketitle
\vspace{-0.5cm}

%%%%%%%%%%%%%%%%%%%%%%%%%%%%%%%%%%%%%%%%%%%%%%%%%%%%%%%%%%%%%%%%%%%%%%%%%%%%%%%%%
\begin{abstract}
We study the exponential functional $\int_0^\infty e^{-\xi_{s-}} \, d\eta_s$ of two one-dimensional independent L\'evy processes $\xi$ and $\eta$, where $\eta$ is a subordinator. In particular, we derive an integro-differential equation for the density of the exponential functional whenever it exists. Further, we consider the mapping $\Phi_\xi$ for a fixed L\'evy process $\xi$, which maps the law of $\eta_1$ to the law of the corresponding exponential functional $\int_0^\infty e^{-\xi_{s-}} \, d\eta_s$, and study the behaviour of the range of $\Phi_\xi$ for varying characteristics of $\xi$. 
Moreover, we derive conditions for selfdecomposable distributions and generalized Gamma convolutions to be in the range. On the way we also obtain new characterizations of these classes of distributions.
\end{abstract}

%%%%%%%%%%%%%%%%%%%%%%%%%%%%%%%%%%%%%%%%%%%%%%%%%%%%%%%%%%%%%%%%%%%%%%%%%%%%%%%%%

%%%%%%%%%%%%%%%%%%%%%%%%%%%%%%%%%%%%%%%%%%%%%%%%%%%%%%%%%%%%%%%%%%%%%%%%%%%%%%%%%

\section{Introduction}

Given two independent L\'evy processes  $(\xi_t)_{t\geq 0}$, $(\eta_t)_{t\geq 0}$ the corresponding {\sl exponential functional} is defined as
\begin{equation}\label{eq:expfunc}
 V:=\int_{(0,\infty)} e^{-\xi_{t-}}d\eta_t,
\end{equation}
provided that the integral converges a.s. Necessary and sufficient conditions for this convergence in terms of the L\'evy characteristics of $(\xi_t)_{t\geq 0}$ and $(\eta_t)_{t\geq 0}$ have been given in \cite{ericksonmaller05}.\\
Exponential functionals of L\'evy processes describe the stationary
distributions of generalized Ornstein-Uhlenbeck (GOU) processes. More detailed, if $\xi_t$ tends to $+\infty$ as $t\to\infty$ almost surely, then the law of $V$ defined in \eqref{eq:expfunc} is the unique stationary distribution of the GOU process
\begin{equation} \label{GOUdef}
V_t=e^{-\xi_t} \left( \int_0^t e^{\xi_{s-}}d\eta_s +V_0  \right),\quad t\geq 0,
\end{equation}
where $V_0$ is a starting random variable, independent of $(\xi,\eta)$, on
the same probability space (cf. \cite[Thm. 2.1]{lindnermaller05}). %Hence, when $V_0$ is chosen to have the same distribution as $V$, then the process $(V_t)_{t\geq 0}$ is strictly stationary.\\

Due to their importance in applications and their complexity, exponential functionals have gained a lot of attention from various researchers over the last 25 years. See e.g. the survey \cite{bertoinyor} or the more recent research papers \cite{PardoPatieSavov, PardoRiveroSchaik} for results on exponential functionals of the form $V=\int_0^\infty e^{-\xi_{s-}} \, ds$. Exponential functionals where $\eta$ is a Brownian motion plus drift have been treated for example in \cite{kuznetsovetal}. The case of general L\'evy processes $\xi$ and $\eta$ has been studied e.g. in our previous papers \cite{BLexpfunc} and \cite{BLMranges}. Nevertheless, for several of the more concrete results in \cite{BLMranges}, the setting was narrowed down to the case where $\xi$ is a Brownian motion plus drift and $\eta$ a subordinator.

Still, in general the distribution of exponential functionals is unknown. E.g. Dufresne
(cf \cite[Equation (16)]{bertoinyor}) showed that $V
\overset{d}= \frac{2}{\sigma^2} G^{-1}_{2a/\sigma^2}$ where
$G_k$ is a Gamma($k,1$) random variable, whenever $\xi$ is a Brownian
motion with variance $\sigma^2$ and drift $a>0$, and $\eta$ is deterministic. Here and in the following $\overset{d}=$ denotes equality in distribution. A few more concrete distributions of specific exponential functionals have been obtained in \cite{GjessingPaulsen}. Further it has been investigated whether exponential functionals belong to certain classes of distributions. So, as shown in \cite{bertoinlindnermaller08}, $V$ is selfdecomposable whenever $\xi$ is spectrally negative, i.e. has no positive jumps. In \cite{BehmeMaejima} conditions are derived under which the exponential functional \eqref{eq:expfunc} is a generalized gamma convolution, where one of the processes is a compound Poisson process.

In this article we focus on the case of exponential functionals as in \eqref{eq:expfunc} when $\xi$ is a general L\'evy process such that $\lim_{t\to \infty}\xi_t=\infty$  and $\eta$ is a subordinator, independent of $\xi$. By \cite[Cor. 1]{BLMranges} this means that $V\geq 0$ a.s. and we have the following relationship between the characteristic triplet $(\gamma_\xi, \sigma^2_\xi, \nu_\xi)$  of $\xi$ and the Laplace exponents $\psi_\eta$ and $\psi_\mu$ of $\eta_1$ and the distribution $\mu$ of $V$, resp., 
\begin{align} \label{eq:relationwithjumps}
 \psi_\eta (u) = & (\gamma_\xi - \frac{\sigma_\xi^2}{2}) u \psi_\mu'(u) + \frac{\sigma^2_\xi}{2} u^2\left( (\psi'_\mu(u))^2 - \psi''_\mu(u) \right) \\
 & + \int_{\RR} \left( e^{\psi_\mu(u) - \psi_\mu(ue^{-y})} - 1 - u \psi'_\mu(u) y \mathds{1}_{|y|\leq 1}\right) \nu_\xi(dy), \quad u >0. \nonumber
\end{align}
Starting from this, we will consider several aspects of exponential functionals. In particular, in Section \ref{sec:density}, we derive an integro-differential equation for the density of the exponential functional (given its existence) which extends a previous result from \cite{CarmonaPetitYor} where $\eta$ was assumed to be deterministic. \\
Since selfdecomposable distributions and generalized Gamma convolutions play an important role in the remainder of the paper, we review them and their connection to exponential functionals in Section \ref{sec:prelim}, which also includes some new results on these classes of distributions. Further, Section \ref{sec:nested} is concerned with the behaviour of the class of distributions of exponential functionals for varying characteristics of $\xi$. In Sections \ref{sec:selfdec} and \ref{sec:ggc} we derive general conditions for selfdecomposable distributions to be given by an exponential functional with predetermined process $\xi$ and also apply these on generalized Gamma convolutions. Finally, Section \ref{sec:proof} contains the proof of Proposition \ref{prop:GGCfactorinBO}.

\subsection*{Notation}

We write $\mu=\cL(X)$ if $\mu$ is the distribution of the random variable $X$. The set of all probability distributions on $\RR$ ($\RR_+$) is denoted by $\cP$ ($\cP^+$).\\
For a real-valued L\'evy process $(\xi_t)_{t\geq 0}$, the {\it characteristic exponent}
is given by its L\'evy-Khintchine formula (e.g. \cite[Thm. 8.1]{sato})
\begin{align} \label{levykhintchine}
\log \phi_\xi(u) &:= \log \EE\left[e^{i u \xi_1 } \right]\\
&= i \gamma_\xi u - \frac{1}{2} \sigma_\xi^2 u^2   + \int_{\RR}
(e^{iux} -1 -i ux  \mathds{1}_{|x|\leq 1}) \nu_\xi(dx), \quad u\in \RR, \nonumber
\end{align}
where $(\gamma_\xi, \sigma_\xi^2, \nu_\xi)$ is the {\it characteristic triplet} of the L\'evy process $\xi$. We refer to \cite{sato} for  further information on L\'evy processes.\\
In the special case of a subordinator $(\eta_t)_{t\geq 0}$, i.e. of a nondecreasing L\'evy process,
we will also use its Laplace transform which we denote as $\LL_{\eta} (u) := \LL_{\eta_1} (u)=\EE [e^{-u\eta_1}]= e^{-\psi_\eta(u)}$, $u\geq 0$, where the {\sl Laplace exponent} $\psi_\eta$ is a {\sl Bernstein function} (BF), i.e.
\begin{equation} \label{eq:BF} 
 \psi_\eta(u)=a_\eta u + \int_{(0,\infty)} (1-e^{-ut})\nu_\eta(dt), \quad u>0,
\end{equation}
with $a\geq 0$ called the {\sl drift} of $\eta$ and a L\'evy measure $\nu_\eta$. 
A thorough introduction to BFs can be found in the monograph \cite{rene-book}. Remark that general BFs as defined in \cite{rene-book} may have an additional constant term, while in this article we restrict on BFs which are Laplace exponents of a probability measure, that is which are zero in zero and hence are of the form \eqref{eq:BF}.  \\
Similarly, the Laplace transform of a random variable $X$ on $\RR_+$ with $\mu=\cL(X)$ is written as $\LL_X(u)=\LL_\mu(u)=\EE [e^{-uX}]= e^{-\psi_X(u)}=e^{-\psi_\mu(u)}$.
Please notice, that this notation of Laplace exponents is different from the previous papers \cite{BLexpfunc, BLMranges} but coincides with the notation used in \cite{rene-book}.

%We write ``$\stackrel{d}{\to}$'' to denote convergence in distribution of random variables, and ``$\stackrel{w}{\to}$'' to denote weak convergence of probability measures.
%We use the abbreviation ``i.i.d.'' for ``independent and identically distributed''. 

As in \cite{BLexpfunc, BLMranges}, given a one-dimensional L\'evy process $(\xi_t)_{t\geq 0}$ drifting to $+\infty$, we will consider the mapping
\begin{align*}
\Phi_\xi^+ : D_\xi^+  & \to \cP^+ ,\\
\cL(\eta_1) & \mapsto \cL \left( \int_0^\infty e^{-\xi_{s-}} \,
d\eta_s \right),
\end{align*}
defined on
\begin{align*}
D_\xi^+ := \{ \cL (\eta_1) : \eta=(\eta_t)_{t\ge 0} \, & \mbox{ one-dimensional subordinator independent of }\xi\\
& \hskip 15mm \mbox{ such that }\int_0^\infty
e^{-\xi_{s-}} \, d\eta_s\mbox{ converges a.s.}\},
\end{align*}
and we denote the range of $\Phi_\xi^+$ by
$$R_\xi^+ := \Phi_\xi^+(D_\xi^+).$$

%%%%%%%%%%%%%%%%%%%%%%%%%%%%%%%%%%%%%%%%%%%%%%%%%%%%%%%%%%%%%%%%%%%%%%%%%%%%%
\section{On the density of the exponential functional} \label{sec:density}
\setcounter{equation}{0}
%%%%%%%%%%%%%%%%%%%%%%%%%%%%%%%%%%%%%%%%%%%%%%%%%%%%%%%%%%%%%%%%%%%%%%%%%%%%%%%%

As already observed in previous articles, it follows directly from \cite[Thm. 1.3]{alsmeyeretal} that the exponential functional $V$ has a pure-type law, i.e. its distribution is either absolutely continuous, continuous singular or a Dirac measure, where the latter can only be obtained if both processes, $\xi$ and $\eta$, are deterministic (c.f. \cite[Prop. 6.1]{BLexpfunc}).\\
Absolute continuity of exponential functionals has been studied in detail in \cite{bertoinlindnermaller08}. For the setting of this paper, \cite[Thm 3.9]{bertoinlindnermaller08} shows in particular, that the exponential functional $V$ as in \eqref{eq:expfunc} is absolutely continuous, whenever the subordinator $\eta$ has a strictly positive drift. Further, in \cite[Cor. 2.5]{kuznetsovetal}, it is shown that the exponential functional $V$ as in \eqref{eq:expfunc} is absolutely continuous with continuous density if $\sigma_\xi>0$.\\
Nevertheless, if $\eta$ and $\xi$ both are compound Poisson processes, examples can be constructed in which $V$ is not absolutely continuous (see \cite{lindnersato09} and Remark \ref{rem:lindnersato} below). \\
The following theorem provides an integro-differential equation fulfilled by the density of $V$ whenever it exists. Notice that for the special case of a deterministic process $\eta_t=t$ this result has been obtained in \cite{CarmonaPetitYor} using a different technique. In particular, case (ii) below is a special case of the results in \cite{CarmonaPetitYor} or similarly of \cite[Thm. 2.3]{PardoRiveroSchaik} and is just kept here for completeness.

\begin{theorem}\label{thm:density}
 Assume that $\xi=(\xi_t)_{t\geq 0}$ is a L\'evy process such that  $\lim_{t\to \infty}\xi_t=\infty$ and with characteristic triplet $(\gamma_\xi, \sigma^2_\xi, \nu_\xi)$ such that $\int_{[-1,1]}|x|\nu_\xi(dx)<\infty$ and set $\gamma_0:= \gamma_\xi - \int_{[-1,1]}x\nu_\xi(dx)$. 
Let $\eta=(\eta_t)_{t\geq 0}$ be a subordinator with drift $a_\eta$ and jump measure $\nu_\eta$, independent of $\xi$ and such that at least one of the processes $\xi$ and $\eta$ is non-deterministic.
\begin{enumerate}
 \item If $\sigma_\xi=0$, $\gamma_0>0$ and $\nu_\xi((0,\infty))=0$, then a density $f(t)$, $t\geq 0$, of $\mu=\Phi_\xi(\cL(\eta_1))$ exists, which is continuous on $\RR_+\setminus \{\frac{a_\eta}{\gamma_0} \}$, and fulfills 
 \begin{eqnarray}\label{eq:densityboundedbelow}
f(t)&=&0, \quad t<\frac{a_\eta}{\gamma_0}, \\ 
(a_\eta - \gamma_0  t) f(t) 
& = & - \int_{\frac{a_\eta}{\gamma_0}}^t \left(\nu_\xi((-\infty, \log\frac{s}{t}))  + \nu_\eta((t-s, \infty)) \right) f(s) ds, \quad t\geq \frac{a_\eta}{\gamma_0}. \nonumber
\end{eqnarray}
\item If $\sigma_\xi=0$, $\gamma_0>0$, $\nu_\xi((0,\infty))>0$, $\nu_\xi((-\infty,0))=0$ and $\nu_\eta\equiv 0$, then a density $f(t)$, $t\geq 0$, of $\mu=\Phi_\xi(\cL(\eta_1))$ exists, which is continuous on $\RR_+\setminus \{\frac{a_\eta}{\gamma_0} \}$, and fulfills 
\begin{eqnarray}\label{eq:densityboundedabove}
f(t)&=&0, \quad t>\frac{a_\eta}{\gamma_0}, \\ 
(a_\eta - \gamma_0  t) f(t) 
& = & \int_t^{\frac{a_\eta}{\gamma_0}} \nu_\xi((\log\frac{s}{t}, \infty)) f(s) ds, \quad t\leq \frac{a_\eta}{\gamma_0}. \nonumber
\end{eqnarray}
\item Otherwise, assume that $\mu=\Phi_\xi(\cL(\eta_1))$ is absolutely continuous (with differentiable density $f(t)$, $t\geq 0$, such that $\lim_{t\to 0} t^2 f(t)=0$ if $\sigma_\xi>0$), then $f$ fulfills $\lambda$-a.e. (with $\lambda$ the Lebesgue measure)
\begin{eqnarray}\label{eq:density}
 \lefteqn{a_\eta f(t)- \left(\gamma_0+\frac{\sigma_\xi^2}{2} \right) t f(t) - \frac{\sigma_\xi^2}{2} t^2 f'(t) }\\
& = &\int_t^\infty \nu_\xi((\log\frac{s}{t}, \infty)) f(s) ds - \int_0^t \left(\nu_\xi((-\infty, \log\frac{s}{t})) + \nu_\eta((t-s, \infty)) \right) f(s) ds, \quad t\geq 0. \nonumber
\end{eqnarray}
\end{enumerate}
Conversely, if $f(t)$, $t\geq 0$, is a probability density which fulfills  \eqref{eq:densityboundedbelow}, \eqref{eq:densityboundedabove} or \eqref{eq:density} $\lambda$-a.e. for some L\'evy characteristics $\gamma_0, \sigma^2_\xi, \nu_\xi, a_\eta$ and $\nu_\eta$, then it is a density of the corresponding exponential functional \eqref{eq:expfunc}.
\end{theorem}
\begin{proof}
Starting from \eqref{eq:relationwithjumps}, multiplying on both sides with $\LL_\mu(u)=e^{-\psi_\mu(u)}$ and dividing once by $u$ we obtain for $u >0$
 \begin{align}\label{eq:densityh1}
 \frac{\psi_\eta (u)}{u}\LL_\mu(u) = & -(\gamma_0 - \frac{\sigma_\xi^2}{2}) \LL'_\mu(u) + \frac{\sigma^2_\xi}{2} u \LL''_\mu(u)  + \int_{\RR} \left( \frac{\LL_\mu(ue^{-y})}{u} - \frac{\LL_\mu(u)}{u} \right) \nu_\xi(dy). 
\end{align}
Now assume that $\mu$ has a density, such that $\LL_\mu(u)=\int_0^\infty e^{-ut} f(t) dt$. Denote the inverse Laplace transform by $\overset{\LL^{-1}}\longrightarrow $, then obviously we have $\LL_\mu(u)\overset{\LL^{-1}}\longrightarrow  f(t)$ $\lambda$-a.e. while (assuming $\lim_{t\to 0} t^2 f(t)=0$ and that $f$ is differentiable) $\lambda$-a.e. we get
\begin{align*}
  \LL'_\mu(u) & \overset{\LL^{-1}}\longrightarrow -tf(t),\\
u \LL''_\mu(u)& \overset{\LL^{-1}}\longrightarrow \frac{d}{dt}(t^2 f(t))= 2tf(t)+t^2 f'(t),\\
 \int_{\RR} \left( \frac{\LL_\mu(ue^{-y})}{u} - \frac{\LL_\mu(u)}{u} \right) \nu_\xi(dy) & \overset{\LL^{-1}}\longrightarrow  \int_t^\infty \nu_\xi((\log\frac{s}{t}, \infty)) f(s) ds - \int_0^t \nu_\xi((-\infty, \log\frac{s}{t}))  f(s) ds,
\end{align*}
where the last line follows from 
\begin{align*}
 \lefteqn{\int_{\RR} \left( \frac{\LL_\mu(ue^{-y})}{u} - \frac{\LL_\mu(u)}{u} \right) \nu_\xi(dy) }\\ 
&= \int_{\RR} \left(\int_0^\infty e^{-ut} \left( \int_0^{t e^y} f(s) ds - \int_0^t f(s) ds  \right) dt \right) \nu_\xi(dy) \\
%&=  \int_0^\infty \left(\int_0^\infty e^{-ut} \left( \int_t^{t e^y} f(s) ds  \right) dt \right) \nu_\xi(dy) - \int_{-\infty}^0 \left(\int_0^\infty e^{-ut} \left( \int_{te^y}^t f(s) ds   \right) dt \right) \nu_\xi(dy) \\
&= \int_0^\infty e^{-ut} \left(\int_t^\infty f(s) \left(\int_{\log\frac{s}{t}}^\infty   \nu_\xi(dy)\right) ds \right)dt - \int_0^\infty e^{-ut} \left(\int_0^t f(s) \left(\int_{-\infty}^{\log\frac{s}{t}} \nu_\xi(dy)\right) ds \right)dt.
\end{align*}
Further for the left hand side of \eqref{eq:densityh1} with $\psi_\eta (u)=a_\eta u + \int_{(0,\infty)} (1-e^{-ut})\nu_\eta(dt)$ we will use that
$$\frac{\int_{(0,\infty)} (1-e^{-ut})\nu_\eta(dt)}{u}\LL_\mu(u)= \int_0^\infty e^{-us} \nu_\eta((s, \infty)) ds \, \LL_\mu(u)\overset{\LL^{-1}}\longrightarrow  \int_0^t \nu_\eta((t-s, \infty)) f(s) ds$$
which is due to the fact that convolutions become multiplications under the Laplace transform. Now, putting all terms together we easily derive \eqref{eq:density}.\\
Observe that in the setting of case (i), it follows from \cite[Lemma 1 and Thm. 1]{BLMranges} that the measure $\mu$ has support $[\frac{a_\eta}{\gamma_0}, \infty)$. Further recall that in this case $\xi$ is a spectrally negative process $\xi$  and hence $\mu$ is selfdecomposable and  has a continuous density on $(\frac{a_\eta}{\gamma_0}, \infty)$ (cf. \cite[Thm. V.2.16]{steutelvanharn}). \\ 
By \cite[Lemma 1 and Thm. 1]{BLMranges} in the setting of case (ii) $\mu$ has support $[0,\frac{a_\eta}{\gamma_0}]$ and otherwise $\mu$ has full support on $[0,\infty)$. Hence we derive the corresponding formulas from \eqref{eq:density}. Existence of a density in case (ii) follows from \cite[Thm. 3.9]{bertoinlindnermaller08}, continuity has been proven in \cite{CarmonaPetitYor}.\\
For the converse assume that $f$ is a density which fulfills \eqref{eq:density}, then reverting the above we see that its Laplace transform fulfills \eqref{eq:relationwithjumps} which yields the claim by \cite[Thm. 3]{BLMranges}.
\end{proof}

\begin{remark}\label{rem:lindnersato} \rm
 In \cite{lindnersato09} the exponential functional $V$ as in \eqref{eq:expfunc} has been studied in the case where $(\eta_t)_{t\geq 0}$ is a Poisson process with jump intensity $v>0$, and $\xi_t= (\log c) N_t$ for $c>1$ and another (independent) Poisson process $(N_t)_{t\geq 0}$ with jump intensity $u>0$. \\
From Theorem \ref{thm:density} above, we observe that in this setting, if a density of $V$ exists, then it fulfills $\lambda$-a.e.
$$v\int_{(t-1)\vee 0}^t f(s) ds = u \int_t^{ct} f(s) ds, \quad t\geq 0$$
or in terms of the cumulative distribution function $F(t)=\int_0^t f(s) ds$ and the parameter $q=\frac{v}{u+v}\in (0,1)$
\begin{equation} \label{eq:selfsimcumulant}
 F(t)= (1-q) F(ct) + q F(t-1), \quad t>0, \quad \text{ where } F(t)=0, \quad  t\leq 0.
\end{equation}
Actually, \eqref{eq:selfsimcumulant} can be shown to hold even if $\mu=\cL(V)$ is not absolutely continuous, by a similar proof as for Theorem \ref{thm:density}. 
Further, from \eqref{eq:selfsimcumulant} we deduce the self-similarity relation
$$\mu=(1-q)\, \mu\circ T_0^{-1} + q\, \mu \circ T_1^{-1}$$
for $\mu$ with weights $\{1-q,q\}$ and 
$$T_0:x\mapsto \frac{x}{c}, \quad T_1:x\mapsto x+1.$$
Remark that $T_1$ is not a contraction and hence $\mu$ is not a self-similar measure in the classical and well-studied sense of \cite{Hutchinson}. \\
Nevertheless, in \cite{lindnersato09}, the authors proved that $\mu$ shares some properties with self-similar measures. In particular, $\mu$ is continuous singular if $c$ is a Pisot-Vijayaraghavan number, but for Lebesgue a.a. $c>1$ there exists $\bar{q}<1$ such that $\mu$ is absolutely continuous for all $q\in (\bar{q},1)$. 
\end{remark}

From the theorem above, we can derive characterizations of densities of selfdecomposable distributions on $\RR_+$ as well as of generalized Gamma convolutions. This will be done in Corollaries \ref{cor:densitySD} and  \ref{cor:densityGGC} below. 
For the moment, we end this section with an example of application for Theorem \ref{thm:density}.

\begin{example}\rm \label{ex:COGARCH}
Let $L=(L_t)_{t\geq 0}$ be a L\'evy process with characteristic triplet $(\gamma_L,\sigma_L^2,\nu_L)$ and set 
\begin{equation*} \label{eq-def-S}
 S_t:=[L,L]_t^{d}=\sum_{0<s\leq t} (\Delta L_s)^2, \quad t\geq 0.
\end{equation*}
Then the COGARCH volatility process with parameters $\beta,\eta, \varphi>0$ driven by $L$ or $S$ is defined as
\begin{equation*}\label{cog}
V_t = e^{-\xi_t}\left(V_0 + \beta\int_{(0,t]} e^{\xi_{s}} \,d s \right),\quad t\ge0,
\end{equation*}
where $V_0$ is a nonnegative random variable, independent of $(L_t)_{t\geq 0}$, and 
\begin{equation*}\label{eq-defX}
\xi_t = \eta t - \sum_{0<s\leq t} \log (1+\varphi \Delta S_s),\quad t\geq 0.
\end{equation*}
As originally shown in \cite[Thm. 3.1]{KLM:2004}, the process defined in \eqref{cog} has a strictly stationary distribution if and only if
\begin{equation*}\label{cogstatcond} \int_{\RR_+} \log(1+\varphi y)\,\nu_S(d y) =\int_{\RR} \log(1+\varphi y^2)\,\nu_L(d y) < \eta\end{equation*}
and in this case, the stationary distribution is given by the distribution of the exponential functional 
$$V=\beta \int_{\RR_+} e^{-\xi_{s}} \,d s.$$ 
Since $\xi$ is spectrally negative by construction, we can apply Theorem \ref{thm:density}(i) (or \cite[Prop. 2.1]{CarmonaPetitYor}) to obtain that $V$ has a density $f(t)$, $t\geq 0$, with  $f(t)=0$ for $t< \frac{\beta}{\eta}$, while $f$ is continuous on $(\frac{\beta}{\eta}, \infty)$ fulfilling 
\begin{eqnarray}\label{eq:densityCOGARCH}
(\beta - \eta  t )f(t)  + \int_{\frac{\beta}{\eta}}^t \nu_S\left(\left(\frac{t-s}{s\varphi}, \infty\right)\right) f(s) ds=0, \quad t\geq \frac{\beta}{\eta}.
\end{eqnarray}
Now, if for example $(S_t)_{t\geq 0}$ is chosen to be a Poisson process with intensity $c>0$, we obtain from \eqref{eq:densityCOGARCH} the following difference-differential equation for the cumulative distribution function $F(t)$ of $V$
$$\frac{\eta t - \beta}{c} F'(t)=F(t)-F(\frac{\beta}{\eta}), \quad t\geq \frac{\beta}{\eta},$$
with $F(t)=0$ for $t< \frac{\beta}{\eta}$.
Similarly, for the common choice of $L$ having standard normally distributed jumps, one derives the recursive formula
$$ f(t) =\frac{2}{\beta-\eta t} \int_{\frac{\beta}{\eta}}^t \left( 1- \phi\left(\sqrt{\frac{t-s}{s\varphi}} \right)\right) f(s) ds, \quad t> \frac{\beta}{\eta},$$
where $\phi$ is the cumulative distribution function of the normal distribution.
%or equivalently
%$$\beta f(t)- \eta  t f(t)  + \int_{\frac{\beta}{\eta}}^t \left(\nu_L(-\infty, -\sqrt{\frac{t-s}{s\varphi}})+\nu_L(\sqrt{\frac{t-s}{s\varphi}}, \infty)  \right)f(s) ds=0, \quad t\geq \frac{\beta}{\eta}. $$
\end{example}

%%%%%%%%%%%%%%%%%%%%%%%%%%%%%%%%%%%%%%%%%%%%%%%%%%%%%%%%%%%%%%%%%%%
\section{(Semi-)Selfdecomposability and Generalized Gamma Convolutions}\label{sec:prelim}
\setcounter{equation}{0}
%%%%%%%%%%%%%%%%%%%%%%%%%%%%%%%%%%%%%%%%%%%%%%%%%%%%%%%%%%%%%%%%%%%%%%%%%%%%%%

We will use the following notations for the classes of infinitely divisible distributions:
\begin{align*}
\ID, \ID^+  \quad & \quad \text{infinitely divisible distributions on }\RR, \RR_+ \text{ (respectively)} \\
\ID_{\log}, \ID^+_{\log} \quad & \quad  \text{infinitely divisible distributions on }\RR, \RR_+ \text{ with finite log-moment} 
\end{align*}
Further the following classes of distributions will be introduced in the next subsections:
\begin{align*}
\SD, \SD^+ \quad & \quad  \text{selfdecomposable distributions on }\RR, \RR_+\\
\SD(c), \SD(c)^+ \quad & \quad  c\text{-decomposable/ semi-selfdecomposable distributions on }\RR, \RR_+\\
\BO \quad & \quad  \text{Goldie-Steutel-Bondesson class, Bondesson's class (on } \RR_+)\\
\T \quad & \quad  \text{Thorin's class, generalized gamma convolutions (on } \RR_+)
\end{align*}

\subsection{Selfdecomposability}

A random variable $X$ (or equivalently a probability measure $\mu$) is called {\it selfdecomposable}, if for all $c\in(0,1)$, there exists a random variable $Y_c$, independent of $X$, such that
\begin{equation}\label{eq-def-SD}
X\overset{d}= cX'+Y_c,\end{equation}
where $X'$ is an independent copy of $X$. In this case we write $\mu=\cL(X)\in \SD$.
Obviously, for distributions on the positive real line, \eqref{eq-def-SD} is equivalent to 
$$\LL_\mu(u)=\LL_\mu(cu)\LL_{\mu_c}(u), \quad u\geq 0, c\in (0,1),$$
or \begin{equation}\label{eq-SD-bernsteindifferenz}
    \psi_\mu(u)-\psi_\mu(cu)=\psi_{\mu_c}(u), \quad u\geq 0, c\in (0,1),
   \end{equation}
where $\mu_c=\cL(Y_c)$.
In particular it is known (cf. \cite[Prop. 5.17]{rene-book}), that every $\mu\in \SD^+$ has a Laplace exponent of the form
\begin{equation} \label{eq:selfdecomposablebernstein}
 \psi_\mu(u)= au + \int_0^\infty (1-e^{-u t}) \frac{k(t)}{t} dt, \quad u\geq 0,
\end{equation}
with $a\geq 0$ called the {\sl drift} of $\mu$ and $k:[0,\infty)\to [0,\infty)$ non-increasing.\\

The following proposition collects characterizations of selfdecomposable distributions in $\cP^+$ which we intend to use in this paper. Most of them are well known. We couldn't find characterization (iv) in this form in the literature, so we give a short instructive proof. Alternatively (iv) is easily seen to be equivalent to the characterization of selfdecomposability in \cite[Thm. V.2.9]{steutelvanharn}. Further characterizations of selfdecomposable distributions can also be found in \cite{maejimalevymatters, satolevymatter, steutelvanharn} and for a.s. positive random variables in the recent article \cite{maischenkscherer} as well as in Corollary \ref{cor:densitySD} below.

\begin{proposition}\label{prop:SD}
 Let $\mu\in \cP^+$ be a probability measure with Laplace exponent $\psi_\mu(u)$, $u\geq 0$. Then the following statements are equivalent.
\begin{enumerate}
\item[{\rm (i)}] $\mu\in \SD^+$.
%\item[{\rm (ii)}] $\frac{\LL_{\mu}(u)}{\LL_{\mu}(cu)}$ is completely monotone as a function of $u>0$ for all $c\in(0,1)$
\item[{\rm (ii)}] $\psi_{\mu_c}(u):=\psi_\mu(u)-\psi_\mu(cu)$ is a Bernstein function for all $c\in(0,1)$.
\item[{\rm (iii)}] $-\psi_{\mu_c}(u)= \psi_\mu(cu) - \psi_\mu(u)$ is a BF for all $c>1$.
\item[{\rm (iv)}] $u\cdot \psi_\mu'(u)$ is a BF.
\item[{\rm (v)}] $\mu=\cL(\int_{(0,\infty)} e^{-t}dX_t)$ for some subordinator $(X_t)_{t\geq 0}$ with $\EE[\log^+(X_1)]<\infty$.
\end{enumerate}
\end{proposition}
\begin{proof}
Equivalence of (i) and (ii) is well known and follows immediately from the definition of selfdecomposability and the fact that $\mu_c$ as in \eqref{eq-def-SD} is infinitely divisible (see e.g. \cite[Prop. 15.5]{sato}).  Further by \cite[Cor. 3.8(iii)]{rene-book} (ii) implies that also $\psi_{\mu_c}(c^{-1}u)=\psi_\mu(c^{-1}u)-\psi_\mu(u)$, $c\in(0,1)$ is a BF, i.e. (iii). The converse can be seen similarly.\\
We continue proving that (ii) implies (iv). Assume (ii), then for all $c\in(0,1)$
$$ \frac{\psi_\mu(u)-\psi_\mu(u-(1-c)u)}{(1-c)}$$
is a BF in $u$. %For $u$ fixed 
%$$\psi_\mu'(u)=\lim_{c\to 1} \frac{\psi_\mu(u)-\psi_\mu(u-(1-c)u)}{(1-c)u}$$
%exists and so 
Thus
$$u\psi_\mu'(u)=\lim_{c\to 1} \frac{\psi_\mu(u)-\psi_\mu(u-(1-c)u)}{(1-c)}$$
is a BF, too (\cite[Cor. 3.8(ii)]{rene-book}), which shows (iv).\\
Now assume (iv) and set 
\begin{equation} \label{eq-SD-bernsteinableitung}
 \psi_X(u):=u\psi'_\mu(u), \quad u\geq 0,
\end{equation}
 then $\psi_X$ is a BF with $\psi_X(0)=0$ and hence there exists a subordinator $(X_t)_{t\geq 0}$ with Laplace exponent $\psi_X$. Now by \cite[Thm. 5.1 (ii)]{BLMranges} (setting $\sigma=0$) this implies that
\begin{equation} \label{eq-SD-integral}
 \mu=\cL \left( \int_{(0,\infty)} e^{-t} dX_t \right).
\end{equation}
Since $\mu$ exists by assumption and therefore the integral has to converge, we obtain $\EE[\log^+(X_1)]<\infty$ and hence (v).\\
Finally, $ \int_{(0,\infty)} e^{-t} dX_t$ is well known and easily seen to be selfdecomposable (see e.g. \cite{bertoinlindnermaller08}) which concludes the proof.
\end{proof}

\begin{remark}\rm
 As already observed in \cite{BLMranges}, Equation \eqref{eq-SD-bernsteinableitung} implies in particular, that $\mu$ and $\cL(X)$ have the same drift and that the L\'evy density of $\mu$ and the L\'evy measure of $X$ are related by
\begin{equation} \label{eq:relationlevymeasures}
 k(t)=\nu_X((t,\infty))
\end{equation}
(see also \cite[Eq. 4.17]{BarndorffNielsen-Shephard}).
\end{remark}

\begin{definition}\label{factors} {\rm Differences of BFs as in (ii) and (iii) of the above proposition will appear frequently in the remaining sections of this article. Hence in the following, we refer to the distributions with Laplace exponent $\psi_{\mu_c}$ ($c\in(0,1)$) or $-\psi_{\mu_c}$ ($c>1$) as $c$-{\sl factor distributions} of the distribution $\mu\in \SD$. Recall that these are always in $\ID$ and that they are uniquely determined since $\mu \in \ID$.\\
In terms of random variables we refer to $Y_c$ as the $c$-{\sl factor} ($c\in(0,1)$) of $X$ if $X\overset{d}=cX'+Y_c$ and we say that $Y_c$ is the $c$-{\sl factor} ($c>1$) for $X$, if $cX\overset{d}=X'+Y_c.$}
\end{definition}

%It is well known, that the class of selfdecomposable distributions is a subclass of $\ID$, i.e. the class of infinitely divisible distributions on $\RR$. 

Further, from Theorem \ref{thm:density} above we obtain the following characterization of densities of distributions in $\SD_+$. The fact that densities of selfdecomposable distributions fulfill an equality like \eqref{eq:densitySD} can also be found in \cite[Thm. V.2.16]{steutelvanharn}. Here we see that actually all solutions to \eqref{eq:densitySD} correspond to distributions in $\SD^+$.

\begin{corollary}\label{cor:densitySD}
 Let $f(t)$ be a probability density with support $[a, \infty)$, $a\geq 0$, which is continuous on $(a, \infty)$. Then $f$ corresponds to a selfdecomposable distribution, if and only if $f$ fulfills
\begin{eqnarray}\label{eq:densitySD}
 (a-t) f(t) + \int_a^t  \nu ((t-s, \infty))  f(s) ds=0, \quad t\geq a,
\end{eqnarray}
for some L\'evy measure $\nu$ such that $\int_0^\infty \log^+(x) \nu(dx)<\infty$.
\end{corollary}
\begin{proof}
Every distribution $\mu\in\SD^+$ which is non-degenerate is absolutely continuous and can be represented as $\mu=\Phi_\xi(\cL(L_1))$ for $\xi_t=t$ and some subordinator $L$ with $\EE[\log^+(L_1)]<\infty$. Further $\supp(\mu)=[a, \infty)$, $a\geq 0$, implies by \cite[Thm. 1(ii)]{BLMranges} that $L$ has drift $a$. Hence by Theorem \ref{thm:density}(i) the density of $\mu$ fulfills \eqref{eq:densitySD}.\\
Conversely, if $f(t)$ is a density with support $[a, \infty)$, $a\geq 0$, which is continuous on $(a, \infty)$ and which fulfills \eqref{eq:densitySD}, then by Theorem \ref{thm:density} it is the density of the exponential functional $\int_{(0,\infty)} e^{-t} dL_t$ for some subordinator $L$ with L\'evy measure $\nu$ and drift $a\geq 0$. Thus we conclude that $\mu\in\SD^+$.
\end{proof}

\subsection{Semi-selfdecomposability}

We say a random variable $X$ (or its probability measure $\mu$) is {\it c-decomposable}, $c\in(0,1)$, or semi-selfdecomposable if \eqref{eq-def-SD} holds for a $c\in(0,1)$ and a random variable $Y_c$ such that $\cL(Y_c) \in \ID$. We write $\SD(c), \SD^+(c)$ for the class of $c$-decomposable distributions on $\RR$ and $\RR_+$, respectively. 
As in the case of selfdecomposable distributions, we refer to the random variable $Y_c$ in \eqref{eq-def-SD} as the {\it c-factor} of $X$. 

By \cite[Prop. 15.5]{sato} it holds $\SD(c)\subset \ID$. 

For probability distributions on $\RR_+$ one can characterize $c$-decomposability in terms of the Laplace exponents. In particular, $\mu\in \SD^+(c)$ if and only if $\psi_{\mu_c}(u)=\psi_\mu(u)-\psi_\mu(cu)$, $u>0$, is a BF. The fact that BFs build a convex cone then implies directly $\SD^+(c) \subseteq \SD^+_{c^n}$ for all $n\in\NN$. More detailed 
\begin{equation} \label{eq:cnfactor}
 \psi_{\mu_{c^n}}(u)=\sum_{i=0}^{n-1} \psi_{\mu_c} (c^i u)
\end{equation}
is the Laplace exponent of the $c^n$-factor of $\mu\in \SD^+(c)$. 
Using this one further obtains for any $\mu\in \SD^+(c)$
$$\psi_\mu (u)= \lim_{n\to \infty} \psi_{\mu_{c^n}}(u) = \lim_{n\to\infty} \sum_{i=0}^{n-1} \psi_{\mu_c}(c^i u)$$
such that
$$\SD^+(c)=\{ \mu \in \cP^+, \mbox{ s.t. } \psi_\mu(u) =  \sum_{i=0}^{\infty} f(c^i u)  \mbox{ for some BF }f\}.
$$

%%%%%%%%%%%%%%%%%%%%%%%%%%%%%%%%%%%%%%%%%%%%%%%%%%%%%%%%%%%%%%%%%%%%%%%%%%%%%%%%%%%%
\subsection{Generalized Gamma Convolutions}
%%%%%%%%%%%%%%%%%%%%%%%%%%%%%%%%%%%%%%%%%%%%%%%%%%%%%%%%%%%%%%%%%%%%%%%%%%%%%%%%%%%%

The class of generalized Gamma convolutions $\T$ is a subclass of the selfdecomposable distributions in $\cP^+$. In particular, every $\mu\in \T$, has a Laplace exponent of the form
\begin{equation} \label{eq:thorinbernstein}
 \psi_\mu(u)= au + \int_{(0,\infty)} (1-e^{-u t}) \frac{k(t)}{t} d t, \quad u\geq 0,
\end{equation}
for some $a\geq 0$ and a completely monotone (CM) function $k:(0,\infty)\to [0,\infty)$. \\
The class of probability distributions whose Laplace transform is of the form \eqref{eq:thorinbernstein} for some $a\geq 0$ with $\frac{k(t)}{t}$ CM is called Goldie-Steutel-Bondesson class or simply Bondesson's class ($\BO$). Its Laplace exponents are referred to as complete Bernstein functions (CBF) and they can always be represented as 
\begin{equation} \label{eq:completebernstein}
 \psi_\mu(u) = a u +  \int_{(0,\infty)} \frac{u}{u + x} d \rho(x), \quad u\geq 0,
\end{equation}
with $a\geq 0$ and a so-called Stieltjes measure $\rho$, that is a measure $\rho$ on $(0,\infty)$ for which $\int_{(0,\infty)} (1+x)^{-1} \rho(dx)<\infty$. For further details and an overview of the existing literature we refer to \cite{maejimalevymatters} and \cite{rene-book}. \\
%Notice that the classes of BFs in \cite{rene-book} also contain Laplace exponents of subprobability measures, while we restrict ourselves on probability measures.\\
Recall that $\BO$ is the smallest class of distributions which contains all mixtures of exponential distributions and is closed under convolutions and  weak limits, while $\T$ is the smallest class that contains all gamma distributions and is closed under convolutions and weak limits. 
Also recall that $\T \subset \BO \subset \ID^+$ and $\T\subset \SD^+ \subset \ID^+$,  but $\SD^+\not\subset \BO$ and $\BO \not\subset \SD^+$.

Generalized Gamma convolutions and distributions in $\BO$ are connected via exponential functionals as shown in the following proposition, which has originally been proven in \cite[Thm. C(iii)]{MaejimaBarndorff}. Nevertheless, we can now give a completely different and shorter proof as we shall do.

\begin{proposition} \label{prop:bondessontoGGC}
 Let $\xi_t=t$. Then 
$$\Phi_\xi(\BO \cap \ID_{\log}) = \T$$
In particular, the distributions in $\BO \cap \ID_{\log}$ with finite Stieltjes measure are mapped surjectively on the generalized Gamma convolutions with $k(0+)<\infty$.
\end{proposition}
\begin{proof}
 Assume $\mu\in \T\subset \SD^+$, then there exists a L\'evy process $X$ with $\cL(X_1)\in \ID^+_{\log}$ such that $\Phi_\xi(X_1)=\mu$, i.e. $X$ and $\mu$ are related via \eqref{eq-SD-bernsteinableitung} or \eqref{eq-SD-integral}.
Hence from \eqref{eq-SD-bernsteinableitung} and \eqref{eq:thorinbernstein}
\begin{align*}
 \psi_X(u)&= a u + u \int_{(0,\infty)} e^{-u t} k(t) dt 
= a u + u \int_{(0,\infty)} e^{-u t} \int_{[0,\infty)} e^{-t x} d \rho (x) dt
\end{align*}
for some unique measure $\rho$ with $\rho(\{0\})= \lim_{t\to \infty} k(t)=0$. Using Tonelli we can proceed
\begin{align*}
 \psi_X(u) &= a u +  \int_{(0,\infty)} u \int_{(0,\infty)} e^{-u t} e^{-t x} dt\, d \rho (x) = a u +  \int_{(0,\infty)} \frac{u}{u + x} d \rho(x).
\end{align*}
Hence $\psi_X(u)$ is a CBF (see e.g. \cite[Remark 6.4]{rene-book}) %,  $\int_0^\infty \frac{1}{1+x}d \nu(x)<\infty$ holds since the above integrals all exist). Hence 
such that $\cL(X_1)\in \BO$ by \cite[Def. 9.1]{rene-book}.\\
Conversely, assume that $X$ is a L\'evy process such that $\cL(X_1)\in \BO\cap \ID_{\log}$. 
Then $\Phi_\xi(\cL(X_1))$ exists and the same computation backwards proves that $\Phi_\xi(\cL(X_1)) \in \T$. \\
The remaining assertion follows directly from an inspection of the above proof.
\end{proof}

From this, we obtain an analogue result to Corollary \ref{cor:densitySD} characterizing the densities of distributions in $\T$.

\begin{corollary}\label{cor:densityGGC}
 Let $f(t)$ be a probability density with support $[a, \infty)$, $a\geq 0$, which is continuous on $(a, \infty)$. Then $f$ is the density of a generalized gamma convolution, if and only if $f$ fulfills
\begin{eqnarray*}
 (a-t) f(t) + \int_a^t f(s) \int_{t-s}^\infty m(x) dx\, ds=0, \quad t\geq a,
\end{eqnarray*}
for some $m(x):(0,\infty)\to [0,\infty)$ which is CM and such that $\int_0^\infty \log^+(x) m(x) dx<\infty$.
\end{corollary}
\begin{proof}
The statement follows similarly to Corollary \ref{cor:densitySD} with the help of Proposition \ref{prop:bondessontoGGC}.
\end{proof}

As mentioned, the $c$-factors of selfdecomposable distributions play an important role for our studies. In the following proposition, which is of interest by its own, we will see, that the GGCs are exactly those distributions in $\SD^+$ whose $c$-factors are all in Bondesson's class. Its proof is postponed to the closing section of this article.

\begin{proposition} \label{prop:GGCfactorinBO}
 Let $\mu\in\T$, then $\mu_c\in\BO$ for all $c>0$, $c\neq 1$. Conversely, if $\mu\in\SD^+$ with either $\mu_c\in\BO$ for all $c\in(0,1)$, or $\mu_c\in\BO$ for all $c>1$, then $\mu\in\T$.
\end{proposition}

Summarizing, we can state the characterizations of the class $\T$ similarly to that of $\SD^+$ in Proposition \ref{prop:SD}.

\begin{corollary}\label{cor:GGC}
 Let $\mu\in \SD^+$ be a probability measure with Laplace exponent $\psi_\mu(u)$, $u>0$. Then the following statements are equivalent.
\begin{enumerate}
\item[{\rm (i)}] $\mu\in \T$.
%\item[{\rm (ii)}] $\frac{\LL_{\mu}(u)}{\LL_{\mu}(cu)}$ is completely monotone as a function of $u>0$ for all $c\in(0,1)$ FIND OUT CHARACTERIZATION
\item[{\rm (ii)}] $\psi_{\mu_c}(u):=\psi_\mu(u)-\psi_\mu(cu)$ is a CBF for all $c\in(0,1)$.
\item[{\rm (iii)}] $-\psi_{\mu_c}(u)= \psi_\mu(cu) - \psi_\mu(u)$ is a CBF for all $c>1$.
\item[{\rm (iv)}] $u\cdot \psi_\mu'(u)$ is a CBF.
\item[{\rm (v)}] $\mu=\cL(\int_{(0,\infty)} e^{-t}dX_t)$ for some subordinator $(X_t)_{t\geq 0}$ with $\EE[\log^+(X_1)]<\infty$ and $\cL(X_1)\in \BO$.
\end{enumerate}
\end{corollary}

%%%%%%%%%%%%%%%%%%%%%%%%%%%%%%%%%%%%%%%%%%%%%%%%%%%%%%%%%%%%%%%%%%%%%%%%%%%%%
\section{Nested ranges} \label{sec:nested}
\setcounter{equation}{0}
%%%%%%%%%%%%%%%%%%%%%%%%%%%%%%%%%%%%%%%%%%%%%%%%%%%%%%%%%%%%%%%%%%%%%%%%%%%%%%%%

In this section, we will consider what happens with the range $R_\xi^+$ when we modify the characteristics of $\xi$. This result has a counterpart in the case when $\xi$ is a Brownian motion (see \cite[Thm. 5]{BLMranges}), although here for some statements we have to restrict on $\SD \cap R_\xi^+$. That this restriction is truly necessary will subsequently be shown in Proposition \ref{prop:notnested}.

\begin{theorem} \label{thm:nested}
Let $(\xi_t)_{t\geq 0}$ be a L\'evy process with characteristic triplet $(\gamma, \sigma^2, \nu)$  and write $R^+(\gamma, \sigma^2, \nu):=R^+_\xi$.\\
Then if $\sigma^2\neq 0$
$$R^+(\gamma, \sigma^2, \nu) = R^+(\gamma/\sigma^2,1, \nu/\sigma^2).$$
Further for $\gamma'\geq \gamma $ it holds
\begin{equation} \label{eq:nested1}
 \SD\cap R^+(\gamma, \sigma^2, \nu) \subseteq  \SD\cap R^+(\gamma', \sigma^2, \nu),
\end{equation}
while assuming that $\nu((0,\infty))=0$ and $\int_{[-1,0)}|x|\nu(dx)<\infty$ we obtain
\begin{equation}\label{eq:nested2}
 R^+(\gamma, \sigma^2, \nu) \subseteq R^+(\gamma', \sigma^2, \lambda\nu)
\end{equation}
for all $\lambda\in(0,1]$ and $\gamma'$ such that $\gamma'-\gamma \geq -(1-\lambda)\int_{[-1,0)}x\nu(dx)$.
\end{theorem}

\begin{proof}
By the L\'evy-It\^o-decomposition we have $\xi_t=\sigma B_t + \tilde{\xi}_t$, where $\sigma=\sqrt{\sigma^2}$ and $(B_t)_{t\geq 0}$ is a standard Brownian motion and independent of $\tilde{\xi}_t$. Hence $(\sigma B_t)_{t\geq 0}\overset{d}= (B_{\sigma^{2}t})_{t\geq 0}$ and thus $(\sigma B_t+ \tilde{\xi}_t)_{t\geq 0}\overset{d}= (B_{\sigma^{2}t} + \tilde{\tilde{\xi}}_{\sigma^2t})_{t\geq 0}$ where $\tilde{\tilde{\xi}}$ has characteristic triplet $(\gamma/\sigma^2, 0 , \nu/\sigma^2)$.\\
This implies that for any subordinator $(\eta_t)_{t\geq 0}$, independent of $\xi$ and with $\cL(\eta_1) \in D_\xi^+$ 
$$\int_{(0,\infty)} e^{-\xi_t} d \eta_t = \int_{(0,\infty)} e^{-(\sigma B_t + \tilde{\xi}_t)} d \eta_t \overset{d}= \int_{(0,\infty)} e^{-(B_{\sigma^{2}t}+ \tilde{\tilde{\xi}}_{\sigma^{2}t})} d \eta_t = \int_{(0,\infty)} e^{-(B_{t}+ \tilde{\tilde{\xi}}_{t})} d \eta_{t/\sigma^2}.$$
Thus $\cL(\eta_{1/\sigma^2}) \in D_{B+\tilde{\tilde{\xi}}}^+$ and $\Phi^+_\xi(\cL(\eta_1))= \Phi^+_{B+\tilde{\tilde{\xi}}}(\cL(\eta_{1/\sigma^2}))$ from which we conclude the first assertion.

Now assume $\mu \in R^+(\gamma, \sigma^2, \nu)\cap \SD$, then by \cite[Thm. 3]{BLMranges}
\begin{align*}
 f_{\gamma}(u) = & (\gamma - \frac{\sigma^2}{2}) u \psi_\mu'(u) + \frac{\sigma^2}{2} u^2\left( (\psi'_\mu(u))^2 - \psi''_\mu(u) \right) \\
 & + \int_{\RR} \left( e^{\psi_\mu(u) - \psi_\mu(ue^{-y})} - 1 - u \psi'_\mu(u) y \mathds{1}_{|y|\leq 1}\right) \nu(dy), \quad u \geq  0,
\end{align*}
is the Laplace exponent of some subordinator, i.e. a BF.  Observe that for $\gamma'\geq \gamma$
\begin{align*}
f_{\gamma'}(u)= & f_{\gamma}(u) + (\gamma'-\gamma)  u \psi'_\mu(u).
\end{align*}
Since the set of BFs is a convex cone (cf. \cite[Cor. 3.8(i)]{rene-book}) and since by assumption $\mu\in\SD^+$ such that $u \psi'_\mu(u)$ is a BF, $f_{\gamma'}(u)$ is again a BF. Hence $\mu\in R^+(\gamma', \sigma^2, \nu)$ by \cite[Thm. 3]{BLMranges}. 

Finally, assume $\mu \in R^+(\gamma, \sigma^2, \nu)$ where $\nu((0,\infty))=0$ and $\int_{[-1,0)}|x|\nu(dx)<\infty$ and set for $\lambda\in(0,1]$
\begin{align*}
 g_{\lambda}(u) = & (\gamma_\lambda - \frac{\sigma^2}{2}) u \psi_\mu'(u) + \frac{\sigma^2}{2} u^2\left( (\psi'_\mu(u))^2 - \psi''_\mu(u) \right) \\
 & + \int_{\RR_-} \left( e^{\psi_\mu(u) - \psi_\mu(ue^{-y})} - 1 \right) \lambda \nu(dy), \quad u\geq  0,
\end{align*}
where $\gamma_\lambda:=\gamma -  \lambda\int_{[-1,0)} x \nu(dx)$,  then $g_1(u)$ is a BF by assumption.
For any $\lambda<1$ we observe that for $u >0$
\begin{align*}
 g_{\lambda}(u) =& g_1(u)+(1-\lambda)\int_{[-1,0)} x \nu(dx)u \psi_\mu'(u) +  (1-\lambda) \int_{\RR_-} \left( 1- e^{\psi_\mu(u) - \psi_\mu(ue^{-y})} \right) \nu(dy).
\end{align*}
Since $\xi$ is spectrally negative, $\mu$ is selfdecomposable and thus $\psi_\mu(ue^{-y})-\psi_\mu(u)$ is a BF for any negative $y$ by Proposition \ref{prop:SD} (it is the Laplace exponent of the $e^{-y}$-factor of $\mu$). Hence $e^{\psi_\mu(u) - \psi_\mu(ue^{-y})}$ is CM and we can write
$$e^{\psi_\mu(u) - \psi_\mu(ue^{-y})} = \int_{(0,\infty)} e^{-ut} \mu_{e^{-y}}(dt).$$
Thus for $u>0$
\begin{align*}
 g_{\lambda}(u) %=& g_1(u)+ (c-1) \int_{\RR_-} \left(\int_{(0,\infty)} (e^{-ut} - 1) \mu_{e^{-y}}(dt) - u \psi'_\mu(u) y \mathds{1}_{|y|\leq 1}\right) \nu(dy)\\
=& g_1(u)+(1-\lambda)\int_{[-1,0)} x \nu(dx)u \psi_\mu'(u)+ (1-\lambda) \int_{\RR_-} \int_{(0,\infty)} \mu_{e^{-y}}(dt)\nu(dy) \left( 1- e^{-ut} \right).
\end{align*}
Since $u \psi_\mu'(u)$ is a BF by Proposition \ref{prop:SD} and since all appearing integrals exist, we conclude that $g_\lambda(u)+(\gamma' - \gamma)u \psi_\mu'(u)$ is again a BF. Hence $\mu\in R^+(\gamma', \sigma^2, \lambda\nu)$ which proves \eqref{eq:nested2}.
\end{proof}

\begin{proposition} \label{prop:notnested}
Let $(\xi_t)_{t\geq 0}$ be a subordinator with drift $a>0$ and jump measure $\nu$ and set $R^+(a,\nu):=R^+_{\xi}$. 
Then for $a'>a$ we have 
$$\SD\cap R^+(a, \nu) \subseteq  \SD\cap R^+(a', \nu),$$
but
$$R^+(a, \nu) \setminus R^+(a', \nu) \neq \emptyset.$$
 \end{proposition}
\begin{proof}
 The first statement has been shown in Theorem \ref{thm:nested}. \\
Let $\mu:=\Phi_{\xi^{(a)}}(\delta_1)$ be the law of $\int_{(0,\infty)} e^{-\xi^{(a)}_t} dt$, then $\mu\in R^+(a, \nu)$ with $\supp\mu=[0,\frac1a]$ in case of a non-deterministic $\xi$ and  $\supp\mu=\{\frac1a\}$ if $\xi$ is deterministic (cf. \cite[Lemma 2.1]{BLMranges}). \\
On the other hand by \cite[Lemma 2.1 and Thm. 2.2]{BLMranges} all distributions in $R^+(a',\nu)$ have support $[0,\infty)$, $[0, \frac{1}{a'}]$ ($\xi$ non-deterministic) or $\{\frac{1}{a'}\}$ ($\xi$ deterministic). Hence $\mu\not\in R^+(a',\nu)$. 
\end{proof}

In case of varying jump heights, nested ranges cannot be expected. To illustrate this, we consider the case of Poisson processes with varying jump height in which we can fully describe the range as we shall do in the following proposition, which also improves the previous result \cite[Prop. 6.3]{BLexpfunc}.

\begin{proposition}
 Assume that $\xi_t= c N_t$ for a Poisson process $N=(N_t)_{t\geq 0}$ with intensity $\lambda$ and some $c>0$. Then
\begin{align}\label{eq-rangePP}
 R_\xi^+& = \{\mu \in \SD_{e^{-c}} \mbox{ with compound exponentially distributed $e^{-c}$-factor}\} \\
&= \{\mu \in \cP^+, \mbox{ s.t. } \psi_\mu(u) = \lim_{n\to\infty} \log\left(\frac{\prod_{k=0}^{n-1} (f(e^{-kc} u) +\lambda)}{\lambda^n}\right) \mbox{ for some BF }f \}.\nonumber
\end{align}
\end{proposition}
\begin{proof}
In the present case \eqref{eq:relationwithjumps} reduces to
 \begin{equation} \label{eq-relationPP}
 \psi_\eta (u) = \lambda e^{\psi_\mu(u) - \psi_\mu(ue^{-c})} - \lambda , \quad u >0.
\end{equation}
Set $\tilde{c}=e^{-c}$, then this is equivalent to
$$ \psi_{\mu_{\tilde{c}}}(u)=\psi_\mu(u) - \psi_\mu(u\tilde{c})= \log\left(\frac{\psi_\eta (u) +\lambda}{\lambda}\right), $$
i.e. $ \psi_{\mu_{\tilde{c}}}(u)$ is the Laplace exponent of a compound exponential distribution  - the distribution of $\eta_T$ for some exponential random variable $T$, independent of $\eta$ - and hence it is the Laplace exponent of an infinitely divisible distribution (cf. \cite[Chapter 3, Thm. 3.6]{steutelvanharn}), i.e. a BF. This proves the first equation in \eqref{eq-rangePP}.  \\
 By iterating and taking limits we further obtain
$$\psi_\mu (u)= \lim_{n\to \infty} \psi_{\mu_{\tilde{c}^n}}(u) = \lim_{n\to\infty} \sum_{k=0}^{n-1} \log\left(\frac{\psi_\eta (\tilde{c}^k u) +\lambda}{\lambda}\right) = \lim_{n\to\infty} \log\left(\frac{\prod_{k=0}^{n-1} (\psi_\eta (\tilde{c}^k u) +\lambda)}{\lambda^n}\right)$$
which proves the second equality in \eqref{eq-rangePP}.  
\end{proof}

\begin{remarks}{\rm 
\begin{enumerate}
\item 
Although for $n\in\NN$ we have $\SD^+(e^{-c})\subseteq \SD^+(e^{-nc})$, the ranges $R^+_{\xi^{(n)}}$ for $\xi_t^{(n)}= nc N_t$ with $(N_t)_{t\in\NN}$ being a Poisson process are in general not nested. In fact, assume that $\mu\in  R_{\xi^{(1)}}^+ \subset \SD^+(e^{-c})\subseteq \SD^+(e^{-nc})$ is given. Then it can be seen from \eqref{eq:cnfactor} that the $e^{-nc}$-factor of $\mu$ has the same distribution as an independent sum of (scaled) compound exponentially distributed random variables. Such sums are in general not compound exponentially distributed. A counterexample can be constructed using the Gamma($k,\theta$) distribution with Laplace transform $\LL(u)= (\frac{\theta}{\theta+u})^{k}$, which is a compound exponential distribution if and only if $k\leq 1$ (cf. \cite[Chapter III, Ex. 5.4]{steutelvanharn}). The convolution of a Gamma($k,\theta$) distribution and a scaled Gamma($k,\theta$) distribution with Laplace transform $\LL(e^{-c} u)= (\frac{\theta}{\theta+e^{-c} u})^{k}$ is no 
compound exponential distribution. This can be seen by applying \cite[Chapter III, Thm. 5.1]{steutelvanharn} and using simple algebra to observe that $\frac{d}{du} (\LL(u)\LL(e^{-c}u))^{-1}$ is not CM.  % while the independent sum of a Gamma($ k_1,\theta$) and a Gamma($k_2,\theta$) distributed random variable is Gamma($ k_1+k_2, \theta$) distributed.
\item Since BFs grow at most linearly (cf. \cite[Cor. 3.8 (viii)]{rene-book}), the above proposition implies that in the given setting $\psi_\mu (u)=o(u^\alpha)$ for any $\alpha>0$. Hence $\psi_\mu$ has zero drift and also no polynomial part (in particular $\mu$ can not be stable).
\end{enumerate}
 }\end{remarks}

%%%%%%%%%%%%%%%%%%%%%%%%%%%%%%%%%%%%%%%%%%%%%%%%%%%%%%%%%%%%%%%%%%%%%%%%%%%%%%%%%%%%%%%%%%%
\section{Selfdecomposable distributions in the range} \label{sec:selfdec}
\setcounter{equation}{0}
%%%%%%%%%%%%%%%%%%%%%%%%%%%%%%%%%%%%%%%%%%%%%%%%%%%%%%%%%%%%%%%%%%%%%%%%%%%%%%%%%%%

In this section, we derive a general criterion for a probability distribution to be in $R_\xi^+$ for a spectrally negative L\'evy process $\xi$. Recall that in this case $R_\xi^+ \subseteq \SD^+$.

\begin{theorem}\label{thm:selfdec}
 Let $\mu\in\SD^+$. Assume that $\xi=(\xi_t)_{t\geq 0}$ is a L\'evy process with characteristic triplet $(\gamma_\xi, \sigma^2_\xi, \nu_\xi)$ such that $\nu_\xi((0,\infty))=0$, $\int_{[-1,0)}|x|\nu_\xi(dx)<\infty$ and $\lim_{t\to \infty}\xi_t=\infty$.\\
Set $\gamma_0:= \gamma_\xi - \int_{[-1,0)}x\nu_\xi(dx)>0$, let $\nu_X$ be the L\'evy measure of the L\'evy process $X$ which is related to $\mu$ via \eqref{eq-SD-integral} and let $\mu_{c}$, $c>1$, be the $c$-factor distribution of $\mu$ as defined in Definition~\ref{factors}.
\begin{enumerate}
 \item If $\sigma_\xi^2=0$, then $\mu\in R_\xi^+$ if and only if
\begin{align} \label{eq:conditionlm1}
 G_1:(0,\infty)&\to[0,\infty)\\
 t&\mapsto \gamma_0  \nu_X((0,t))- \int_{\RR_-} \mu_{e^{-x}}((0,t)) \nu_\xi(dx) \nonumber 
\end{align}
is non-decreasing.
In this case $\mu=\cL(\int_0^\infty e^{-\xi_{t-}}d\eta_t)$, where $\eta$ is a subordinator, independent of $\xi$, with L\'evy measure $\nu_\eta(dt)=dG(t)$ and drift $a_\eta=\gamma_0 a\geq 0$ where $a\geq 0$ denotes the drift of $\mu$.
\item If $\sigma_\xi^2>0$, assume that $\nu_\xi(\RR_-)<\infty$ and $\nu_X(\RR_+)<\infty$. Then $\mu\in R_\xi^+$ if and only if $\mu$ has zero drift and $\nu_X$ has a density $g(t), t\geq 0,$ such that
\begin{equation}\label{eq:propdichtevonX}
 \lim_{t\to \infty} tg(t) = \lim_{t\to 0} t g(t) = 0,
\end{equation}
and such that
\begin{align} \label{eq-conditionlm2}
 G_2:(0,\infty)&\to[0,\infty)\\
 t&\mapsto  (\gamma_0 + \sigma^2_\xi \nu_X(\RR_+))\int_0^t g(u) du + \frac{\sigma_\xi^2 }{2} tg(t)- \frac{\sigma_\xi^2 }{2} \int_0^t (g\ast g)(u) du \nonumber \\ & \quad \quad - \int_{\RR_-} \mu_{e^{-y}}((0,t)) \nu_\xi(dy) \nonumber
\end{align}
is non-decreasing.
In this case $\mu=\cL(\int_0^\infty e^{-\xi_{t-}}d\eta_t)$, where $\eta$ is a subordinator, independent of $\xi$, with L\'evy measure $\nu_\eta(dt)=dG(t)$ and zero drift.
\end{enumerate}
\end{theorem}

\begin{proof}
Observe that $\gamma_0>0$, since $\EE[\xi_1]>0$ where
$$ \EE[\xi_1] = \gamma_\xi + \int_{(-\infty,-1)} x\nu_\xi(dx) =  \gamma_0 + \int_{\RR_-}x\nu_\xi(dx)= \gamma_0 - \int_{\RR_-}|x|\nu_\xi(dx).$$
By \cite[Thm. 3]{BLMranges} a probability distribution $\mu\in \cP^+$ is in $R_\xi^+$ for the given $\xi$ if and only if 
\begin{align*} 
 f (u):= & \left(\gamma_\xi - \frac{\sigma^2_\xi}{2}\right) u \psi_\mu'(u) + \frac{\sigma^2_\xi}{2} u^2\left( (\psi'_\mu(u))^2 - \psi''_\mu(u) \right)
\\ & + \int_{\RR_-} \left( e^{-(\psi_\mu(ue^{-y})-\psi_\mu(u))} - 1 - u\psi_\mu'(u) y \mathds{1}_{|y|\leq 1}\right) \nu_\xi(dy) 
%& = : a \psi_X (u) + \int_{\RR} \left( e^{-g(u,y)} - 1\right) \nu_\xi(dy) 
\end{align*}
defines a BF. Since $\mu\in \SD^+$, the functions $\psi_X(u)=u\psi'_\mu(u)$ and $-\psi_{\mu_c}(u) = \psi_\mu(cu) - \psi_\mu(u)$, $c>1$, are again BFs by Proposition \ref{prop:SD} and 
\begin{align*} 
 f (u) %= & \, \gamma_\xi \psi_X(u) + \frac{\sigma^2_\xi}{2} \left( (\psi_X(u))^2 - u  \psi'_X(u) \right) + \int_{\RR_-} \left( e^{-\psi_{\mu_{e^{-y}}}(u)} - 1 - \psi_X(u) y \mathds{1}_{|y|\leq 1}\right) \nu_\xi(dy) \\
= &\,  \gamma_0 \psi_X(u) + \frac{\sigma^2_\xi}{2} \left( (\psi_X(u))^2 - u  \psi'_X(u) \right) + \int_{\RR_-} \left( \exp(\psi_{\mu_{e^{-y}}}(u)) - 1 \right) \nu_\xi(dy).
\end{align*}
As $\mu_c$ is the $c$-factor of $\mu$ we have $e^{\psi_{\mu_{c}}(u)} = \int_{[0,\infty)} e^{-ut} \mu_{c}(dt)$, and therefore 
\begin{align}\label{eq:exponentinsd} 
 f (u)= &\,  \gamma_0 \psi_X(u) + \frac{\sigma^2_\xi}{2} \left( (\psi_X(u))^2 - u  \psi'_X(u) \right) + \int_{(0,\infty)} (e^{-ut}-1)  \int_{\RR_-} \mu_{e^{-y}}(dt) \nu_\xi(dy).
\end{align}
Now assume that $\sigma_\xi^2=0$ and let $a\geq 0$ denote the drift of $\mu$, then it follows via \cite[Lemma 1 and Thm. 1]{BLMranges} that $X$ has drift $a$ such that
$$\psi_X (u)= a u + \int_{(0,\infty)} \left(1- e^{-uy} \right) \nu_X(dy),$$
and inserting this in \eqref{eq:exponentinsd} we obtain
\begin{align*} 
 f(u)= &\,  \gamma_0 a u  + \int_{(0,\infty)} (1-e^{-ut})  [\gamma_0 \nu_X(dt) -  \int_{\RR_-} \mu_{e^{-y}}(dt) \nu_\xi(dy)].
\end{align*}
For $f$ to be a BF it is now necessary and sufficient that $\nu_\eta$ defined via 
$$\nu_\eta(dt):= \gamma_0 \nu_X(dt) -  \int_{\RR_-} \mu_{e^{-y}}(dt) \nu_\xi(dy)$$
 is a L\'evy measure, which holds if and only if $G_1$ is non-decreasing.\\
%The remaining assertion follows now directly from \cite[Thm. 3]{BLMranges}.\\
In the case that $\sigma_\xi^2>0$ first observe that from \cite[Lemma 1 and Thm. 1]{BLMranges} we know that supp$\mu=[0,\infty)$ which implies that $\mu$ has drift $0$ and so does $X$.
Further under the assumption that $\nu_X(\RR_+)<\infty$ we obtain as in the proof of \cite[Thm. 7]{BLMranges} that
\begin{equation}\label{eq:selfdecpsisquare}
 (\psi_X (u))^2= \int_{(0,\infty)} (1-e^{-ut})[2\nu_X(\RR_+) \nu_X - \nu_X\ast \nu_X](dt).
\end{equation}
Now suppose $\mu\in R_\xi^+$, then $f$ is a BF, i.e. $f(u)=bu + \int_{(0,\infty)} (1-e^{-ut}) \nu(dt)$,  and we obtain from \eqref{eq:exponentinsd}
\begin{align*}
 \frac{\sigma^2_\xi}{2} u  \psi'_X(u) &= -bu + \int_{(0,\infty)} (1-e^{-ut}) \rho_1(dt) - \int_{(0,\infty)} (1-e^{-ut}) \rho_2(dt)
\end{align*}
where
\begin{align*}
 \rho_1(dt)&:= (\gamma_0 + \sigma^2_\xi \nu_X(\RR_+)) \nu_X(dt) + \int_{\RR_-} \mu_{e^{-y}}(dt) \nu_\xi(dy) \\   \rho_2(dt)&:= \nu(dt)+ \frac{\sigma^2_\xi}{2} \nu_X\ast \nu_X(dt)
\end{align*}
Proceeding as in the proof of \cite[Thm. 7(i)]{BLMranges} this shows $b=0$ and that $\nu_X$ has the density 
$$g(t)=\frac{2}{\sigma_\xi^2 t} (\rho_1(t, \infty)- \rho_2(t,\infty)), \quad t> 0.$$
 Since $\nu_\xi(\RR_-)<\infty$ and $\nu_X(\RR_+)<\infty$, similarly to the argumentation in \cite[Thm. 7(i)]{BLMranges}, it follows that \eqref{eq:propdichtevonX} holds and finally that
\begin{align*}
 \nu(dt)= (\gamma_0 + \sigma^2_\xi \nu_X(\RR_+))g(t) dt + \frac{\sigma_\xi^2}{2} d(tg(t))- \frac{\sigma_\xi^2 }{2} (g\ast g)(t) dt - \int_{\RR_-} \mu_{e^{-y}}(dt) \nu_\xi(dy).
\end{align*}
Thus, if $\mu\in R_\xi^+$, then $\nu(dt)$ has to be a L\'evy measure, which proves that $G_2$ is non-decreasing. Conversely, if $G_2$ is non-decreasing, define a subordinator $\eta$ with L\'evy measure $\nu(dt)=dG(t)$ and zero drift, then reverting the above, it follows from \cite[Thm. 3]{BLMranges} that $\mu\in R_\xi^+$.
\end{proof}

%\begin{example}\rm
% Assume $\mu=\cL(\int_{(0,\infty)} e^{-t} dX_t)$ with $\cL(X_1)\in \SD^+ \cap \ID_{\log}$, or equivalently 
%$\mu\in L_1^+:= \Phi_{t}(\Phi_{t}(\ID^+_{\log^2}))$ where $\ID^+_{\log^2}=\{\rho \in \ID^+: \int_{\RR_+} (\log^+(x))^2 \rho(dx)<\infty\}$. Then $\nu_X$ has a density $k(t)/t$ with $k$ non-increasing. Further $\psi_Y(u):=u\psi_X'(u)$, as it appears in \eqref{eq:exponentinsd}, is again a BF corresponding to a subordinator $Y$ with L\'evy measure $\nu_Y$. Thus, in this case, assuming the  general setting of Theorem \ref{thm:selfdec}, $\sigma_\xi^2>0$ and $\nu_X(\RR_+)<\infty$, it follows via \eqref{eq:exponentinsd} and \eqref{eq:selfdecpsisquare} that $\mu\in R_\xi^+$ if and only if $\mu$ has zero drift and
%\begin{align} \label{eq:conditionlm3}
% G_3:(0,\infty)&\to[0,\infty)\\
% t&\mapsto (\gamma_0 + \sigma^2_\xi \nu_X(\RR_+))\int_{(0,t)}\frac{k(s)}{s} ds  - \frac{\sigma_\xi^2 }{2} \int_{(0,t)}\int_{(0,s)}\frac{k(s-u)}{s-u}\frac{k(u)}{u} du ds   \nonumber \\ & \quad \quad - \frac{\sigma_\xi^2 }{2} \nu_Y ((0,t)) - \int_{\RR_-} \mu_{e^{-y}}((0,t)) \nu_\xi(dy) \nonumber 
%\end{align}
%is non-decreasing.
%\end{example}

\begin{example} \rm
Consider the COGARCH volatility process as introduced in Example \ref{ex:COGARCH}. In this case the process $\xi$ has no gaussian part, L\'evy measure $\nu_\xi=T(\nu_S)$ for the transformation  $T:s\mapsto - \log (1+\varphi s)$ and $\gamma_0=\eta>0$.\\
Since the integrating process in the case of the COGARCH is deterministic $t\mapsto \beta t$, its L\'evy measure is zero and we conclude from Theorem \ref{thm:selfdec}(i) above that the measure $\mu\in \SD^+$, which is the stationary distribution of the COGARCH volatility, has to have drift $a=\frac{\beta}{\eta}$ and that it has to fulfill
\begin{equation} \label{eq:COGARCHmeasures}
 \eta \nu_X(dt)= \int_{\RR_-}\mu_{e^{-x}}(dt) \nu_\xi(dx)= \int_{\RR_+}\mu_{1+\varphi s}(dt) \nu_S(ds),
\end{equation}
where $X$ is connected to $\mu$ via \eqref{eq-SD-bernsteinableitung}.\\
Observe that it follows directly from this, that $$k(0+)= \nu_X(\RR_+)= \eta^{-1}\nu_S(\RR_+),$$ where $k(t), t>0,$ is the factor of the L\'evy density of $\mu$ as in \eqref{eq:selfdecomposablebernstein}.\\
Assuming e.g. that $(S_t)_{t\geq 0}$ is a Poisson process with intensity $c>0$ , we further obtain from \eqref{eq:COGARCHmeasures} that
$$\eta \nu_X(dt) =  c \mu_{1+\varphi}(dt),$$
where $\mu_{1+\varphi}$ has the Laplace exponent $\psi_\mu((1+\varphi)u)-\psi_\mu(u)$. Hence in this case, with   \eqref{eq:selfdecomposablebernstein} and \eqref{eq-SD-bernsteinableitung}
%\begin{align*}
% \log\left(\frac{\eta}{c} \int_{(0,\infty)} e^{-ut} \nu_X(dt)\right) %&= -(\psi_\mu((1+\varphi)u)-\psi_\mu(u))\\
%&= -\left(\frac{\beta}{\eta}\varphi u + \int_{(0,\infty)} (1-e^{-\varphi ut}) e^{-ut} \frac{k(t)}{t} dt\right)
%\end{align*}
%from which 
one can deduce the following equation for the L\'evy density $m(t)=k(t)/t$, $t>0$, of $\mu$, 
$$\frac{\eta}{c} \int_{(0,\infty)} e^{-ut} t dm(t) + \frac{\eta}{c} \int_{(0,\infty)} e^{-ut} m(t) dt= - \exp\left(- \frac{\beta}{\eta}\varphi u - \int_{(0,\infty)} (1-e^{-\varphi ut}) e^{-ut} m(t) dt\right).$$
\end{example}\medskip

\begin{example} \rm
 Assume $\mu$ is positive strictly stable with index $\alpha\in(0,1)$, i.e. $\psi_\mu(u)=cu^\alpha$, for some $c>0$ and let $(\xi_t)_{t\geq 0}$ be a L\'evy process without gaussian part and which fulfills the assumptions of Theorem \ref{thm:selfdec}. Then 
$\mu \in R_\xi^+$ if and only if
\begin{align*}
 \nu(dt)%& := \gamma_0 \nu_X(dt) -  \int_{(0,\infty)} \mu_{e^{-x}}(dt)\nu_{\xi}(dx)\\
& = \gamma_0 \frac{c \alpha^2}{\Gamma(1-\alpha)} t^{-(1+\alpha)} dt -  \int_{(0,\infty)} \mu_{e^{-x}}(dt)\nu_{\xi}(dx)
\end{align*}
defines a L\'evy measure. In particular observe that $\mu_{e^{-x}}$ has Laplace exponent $cu^\alpha(e^{-\alpha x}-1)$ and hence $\nu_Y(dt):= \int_{(0,\infty)} \mu_{e^{-x}}(dt)\nu_{\xi}(dx)$ can be interpreted as the L\'evy measure of $(Y_t)_{t\geq 0}$ where $Y_t=S_{\tilde{\xi}_t}$, with $S=(S_t)_{t\geq 0}$ a strictly $\alpha$-stable subordinator with  $\psi_S(u)=u^\alpha$ and $\tilde{\xi}$ a pure-jump subordinator with L\'evy measure $\nu_{\tilde{\xi}}=T(\nu_\xi)$ for the transformation $T:x\mapsto c(e^{-\alpha x}-1)$ (see e.g. \cite[Thm. 30.1]{sato}).
\end{example}

%%%%%%%%%%%%%%%%%%%%%%%%%%%%%%%%%%%%%%%%%%%%%%%%%%%%%%%%%%%%%%%%%%%%%%%%%%%%%%%%%%%%%%%
\section{GGCs in the range}\label{sec:ggc}
\setcounter{equation}{0}
%%%%%%%%%%%%%%%%%%%%%%%%%%%%%%%%%%%%%%%%%%%%%%%%%%%%%%%%%%%%%%%%%%%%%%%%%%%%%%%%%%%

There exist several examples of exponential functionals whose distributions are generalized Gamma convolutions. Just recall Proposition \ref{prop:bondessontoGGC} or the example mentioned in the introduction, which states that $\int_{(0,\infty)} e^{-(\sigma B_t + at)} dt$ has an inverse Gamma distribution which is a GGC, where $(B_t)_{t\geq 0}$ is a Brownian motion and $\sigma, a>0$. 
Further explicit examples of exponential functionals whose distributions are generalized Gamma convolutions can also be found in \cite{BehmeMaejima} and \cite{BehmeBondesson}.\\
As generalized Gamma convolutions are selfdecomposable, one can also directly transfer the results from the last section to obtain conditions on GGCs to be in the range $R_\xi$ for a given process $\xi$. Together with the results in Section~\ref{sec:prelim} this then yields the following example.

\begin{example}\rm 
Let $\mu\in \T$ have the Laplace exponent \eqref{eq:thorinbernstein} with $a\geq 0$ and $k(0+)<\infty$, $k'(0+)>-\infty$ and $k(t)\not\equiv 0$. Then by Corollary \ref{cor:GGC} the L\'evy measure $\nu_X(dt)$ of the L\'evy process $X$ which is related to $\mu$ via \eqref{eq-SD-integral} has a density $m(t), t\geq 0$, which is CM, that is $\nu_X((0,t))= \int_{(0,t)} m(s) ds$.  \\
Assume that $\xi=(\xi_t)_{t\geq 0}$ is a L\'evy process with characteristic triplet $(\gamma_\xi, 0, \nu_\xi)$ such that $\nu_\xi((0,\infty))=0$, $\int_{[-1,0)}|x|\nu_\xi(dx)<\infty$, $\nu_\xi\not \equiv 0$  and $\lim_{t\to \infty}\xi_t=\infty$.\\
Set $\gamma_0:= \gamma_\xi - \int_{[-1,0)}x\nu_\xi(dx)>0$  and let $\mu_{c}$, $c>1$, be the $c$-factor distribution of $\mu$ as defined in Definition~\ref{factors}, then by Theorem \ref{thm:selfdec} we have $\mu\in R_\xi^+$ if and only if
\begin{align*} 
 G_1:(0,\infty)&\to[0,\infty)\\
 t&\mapsto \gamma_0  \int_{(0,t)} m(s) ds - \int_{\RR_-} \mu_{e^{-x}}((0,t)) \nu_\xi(dx) \nonumber 
\end{align*}
is non-decreasing. \\
By Proposition \ref{prop:GGCfactorinBO} the $c$-factor distributions of $\mu$ are in $\BO$. Further, for $c>1$, they have drift $a_c:=a(c-1)$ and their CM L\'evy densities are given by
$$g_c(t) = \frac{k(c^{-1}t)-k(t)}{t} = t^{-1} \nu_X((c^{-1} t, t]) = t^{-1}\int_{(c^{-1} t, t]}m(s) ds, \quad t>0, $$
(compare the proof of Proposition \ref{prop:GGCfactorinBO}) where the second equality follows from \eqref{eq:relationlevymeasures}. Further, by l'Hospital's rule $g_c(0+)<\infty$, since $k(0+)<\infty$ and $|k'(0+)|<\infty$. Therefore the L\'evy densities $g_c$ are integrable, which implies that the $\mu_c$ are compound Poisson distributed, as it would have followed similarly from \cite[Thm. 6.1]{bondesson81}. Hence $\mu_c=\cL(a_c+ \sum_{i=1}^N Y_i^c)$, where $N\sim$ Poisson$(\lambda_c)$ and where the random variables $Y_i^c$ are i.i.d. with densities $\lambda_c^{-1} g_c(t),$ $t>0$, with $\lambda_c^{-1}:= \int_{(0,\infty)}g_c(t) dt$.    \\
Therefore $\mu_c$ has the density
$$e^{-\lambda_c} \sum_{n=1}^\infty \frac{\lambda_c^n}{n!} (\lambda_c^{-1} g_c(t-a_c))^{\ast n} = e^{-\lambda_c} \sum_{n=1}^\infty \frac{(g_c(t-a_c))^{\ast n}}{n!}, \quad t>a_c,$$
and an atom of mass $e^{-\lambda_c}$ in $a_c$. This yields that $a=0$ is necessary for $\mu$ to be in the range, because otherwise $G_1$ has negative jumps. \\
Now for $a=0$ the term $\int_{\RR_-} \mu_{e^{-x}}((0,t)) \nu_\xi(dx)$ is differentiable and the function $G_1(t)$, $t>0$, as above, is non-decreasing if and only if for all $t>0$
$$\frac{dG_1(t)}{dt} = \gamma_0 m(t) - \int_{\RR_-} \exp(-\lambda_{e^{-x}}) \sum_{n=1}^\infty \frac{(g_{e^{-x}}(t))^{\ast n}}{n!} \nu_\xi(dx)\geq 0.$$
For example, assume that $\mu$ is a Gamma$(k,\theta)$ distribution. Then it has zero drift and its L\'evy density is given by $k t^{-1} e^{-\theta t}$ (cf. \cite[Ex. 8.10]{sato}) such that it fulfills the above assumptions. Further we deduce $m(t)=k \theta e^{-\theta t}$, 
$$g_c(t)=k\cdot \frac{e^{-c^{-1}\theta t}-e^{-\theta t}}{t}, \quad \text{and} \quad \lambda_c=k \log c.$$
Thus
\begin{align*}
\frac{dG_1(t)}{dt} &=  \gamma_0 m(t) - \int_{\RR_-} e^{kx}  \sum_{n=1}^\infty \frac{(g_{e^{-x}}(t))^{\ast n}}{n!} \nu_\xi(dx)\\
&\leq \gamma_0 k \theta e^{-\theta t} - \int_{\RR_-} e^{kx}  g_{e^{-x}}(t) \nu_\xi(dx)\\
%&= \gamma_0 k \theta e^{-\theta t}-  \int_{\RR_-} e^{kx} k\cdot \frac{e^{-e^x \theta t}-e^{-\theta t}}{t} \nu_\xi(dx)\\
&= ke^{-\theta t} \left(\gamma_0 \theta -  \int_{\RR_-} e^{kx} \cdot \frac{e^{\theta t(1-e^x)}-1}{t} \nu_\xi(dx)\right),
\end{align*}
which becomes negative for large $t$, since $\nu_\xi \not\equiv 0$. Therefore in this case we have shown Gamma$(k,\theta) \not\in R_\xi^+$.
\end{example}

Even in the case that $\xi$ has no jumps but a gaussian part, many GGCs can not be in the range as shown in the following.  

\begin{proposition}
 Let $\xi_t=\sigma B_t+at$, $a, \sigma >0$, and let $\mu\in \T$ have the Laplace exponent \eqref{eq:thorinbernstein} with $k(0+)<\infty$ and $k(t)\not\equiv 0$. Then $\mu\notin R_\xi^+$.
\end{proposition}
\begin{proof}
 Let $\mu\in \T$ with $k(0+)<\infty$ be given and define the subordinator $X$ via  \eqref{eq-SD-bernsteinableitung} or \eqref{eq-SD-integral}. Then from Proposition \ref{prop:bondessontoGGC} we know that $\cL(X)\in \BO$ with finite Stieltjes measure and as such it has a Laplace exponent of the form
$$\psi_X(u)= b u +  \int_0^\infty (1-e^{-ut})m(t) dt$$
where $m(t)$ is CM and integrable. From \cite[Thm. 7]{BLMranges} we know that if $\mu\in R_\xi^+$, then necessarily $b=0$. Further from \cite[Remark 7(ii)]{BLMranges} it follows that if $\mu\in R_\xi^+$, then
$$\left(a + \sigma^2 \int_0^\infty m(t) dt  + \frac{\sigma^2}{2} \right) m(t) +   \frac{\sigma^2}{2} t m'(t) - \frac{\sigma^2}{2} (m\ast m)(t) \geq 0, \quad \forall t>0.$$
Since $m(t)$ is CM, it holds
$$m(t)=\int_{[0,\infty)} e^{-\lambda t} d\rho(\lambda)$$
for some measure $\rho$ with $\rho(\{0\})=\lim_{t\to\infty}m(t)=0$. Hence
$$m'(t)= - \int_{(0,\infty)} \lambda e^{-\lambda t} d \rho(\lambda), \quad \int_{(0,\infty)} m(t) dt =\int_{(0,\infty)} \lambda^{-1} d\rho(\lambda)<\infty,$$
and 
\begin{align*}
  (m\ast m)(t) &= \int_0^t m(t-s)m(s) ds %= \int_0^\infty \int_0^\infty e^{-\lambda t} \int_0^t e^{(\lambda - \zeta)s} ds d\rho(\zeta)  d\rho(\lambda)   \\
= \int_{(0,\infty)} \int_{(0,\infty)} \frac{e^{-\zeta t}- e^{-\lambda t}}{\lambda-\zeta} d\rho(\zeta) d\rho(\lambda).
    \end{align*}
So for $\mu\in R_\xi^+$ it is necessary that
\begin{align*}
 \left(a + \sigma^2 \int_{(0,\infty)} \lambda^{-1} d\rho(\lambda)  + \frac{\sigma^2}{2} \right) & \int_{(0,\infty)} e^{-\lambda t} d\rho(\lambda) -   \frac{\sigma^2}{2} t \int_{(0,\infty)} \lambda e^{-\lambda t} d \rho(\lambda) \\
& \quad - \frac{\sigma^2}{2} \int_{(0,\infty)} \int_{(0,\infty)} \frac{e^{-\zeta t}- e^{-\lambda t}}{\lambda-\zeta} d\rho(\zeta) d\rho(\lambda) \geq 0, \quad \forall t>0
\end{align*}
or equivalently
\begin{align} \label{eq:Tinrangecondition}
 & \frac{1}{t} \int_{(0,\infty)} \left(a + \sigma^2 \int_{(0,\infty)} u^{-1} d\rho(u)  + \frac{\sigma^2}{2} \right) e^{-\lambda t} d\rho(\lambda) -  \int_{(0,\infty)} \frac{\sigma^2}{2} \lambda e^{-\lambda t}d\rho(\lambda) \\
& \hspace{3cm}\quad \quad \quad \quad \geq     \frac{1}{t} \int_{(0,\infty)}  \int_{(0,\infty)} \frac{\sigma^2}{2}\frac{e^{-\zeta t}- e^{-\lambda t}}{\lambda-\zeta} d\rho(\zeta) d\rho(\lambda) , \quad \forall t>0. \nonumber
\end{align}
The term on the RHS of \eqref{eq:Tinrangecondition} is non-negative, for the left hand side we observe that by dominated convergence
\begin{align*}
\lefteqn{ \lim_{t\to \infty}  \int_{(0,\infty)}  \left( \frac{a + \sigma^2 \int_{(0,\infty)} u^{-1} d\rho(u)  + \frac{\sigma^2}{2}}{t} -  \frac{\sigma^2}{2} \lambda \right) e^{-\lambda t} d\rho(\lambda) } \\
 &= \int_{(0,\infty)} \lim_{t\to \infty}  \left( \frac{a + \sigma^2 \int_{(0,\infty)} u^{-1} d\rho(u)  + \frac{\sigma^2}{2}}{t} -  \frac{\sigma^2}{2} \lambda \right) e^{-\lambda t} d\rho(\lambda)\\
&< 0
\end{align*}
in contradiction to \eqref{eq:Tinrangecondition}. This proves the proposition.
\end{proof}

%%%%%%%%%%%%%%%%%%%%%%%%%%%%%%%%%%%%%%%%%%%%%%%%%%%%%%%%%%%%%%%%%%%%%%%%%%%%%%%%%%%%%%%%%%%%
\section{Proof of Proposition \ref{prop:GGCfactorinBO}}\label{sec:proof}
\setcounter{equation}{0}
%%%%%%%%%%%%%%%%%%%%%%%%%%%%%%%%%%%%%%%%%%%%%%%%%%%%%%%%%%%%%%%%%%%%%%%%%%%%%%%%%

For the proof of Proposition \ref{prop:GGCfactorinBO} we need the following two simple lemmata.

\begin{lemma}\label{lem:cm1}
 Let $\lambda>0$ be constant, then 
$$f(x)=\frac{1-e^{-\lambda x}}{x}, \quad x>0,$$
is completely monotone.
\end{lemma}
\begin{proof}
 Obviously $f$ is infinitely often continuously differentiable and it holds $f(x)>0$, $x>0$. Further it can be shown by an elementary induction, that the $n$-th derivative of $f$ is given by
\begin{equation}\label{eq:derivativeforbondesson}
 f^{(n)}(x)=(-1)^{n} n! e^{-\lambda x} x^{-(n+1)} \left(e^{\lambda x} - \sum_{k=0}^n \frac{(\lambda x)^k}{k!}\right).
\end{equation}
It follows from the series representation of the exponential function, that the term in the brackets in \eqref{eq:derivativeforbondesson} is positive. Hence $(-1)^{n}f^{(n)}(x)\geq 0$, $x>0$, for all $n$ as we had to show.
\end{proof}

\begin{lemma} \label{lem:cm2}
 Let $k(x)$, $x>0$, be completely monotone and let $c>1$ be some constant. Then
$$f(x)=\frac{k(x)-k(c x)}{x}$$
is completely monotone.
\end{lemma}
\begin{proof}
Assume first that $k(x)=e^{-\lambda x}$ for some $\lambda >0$. Then
$$f(x)=\frac{e^{-\lambda x}-e^{-\lambda x c}}{x} =e^{-\lambda x}\frac{1-e^{-\lambda x (c-1)}}{x}$$
is CM since $e^{-\lambda x}$ and $x^{-1}(1-e^{-\lambda x (c-1)})$ are CM by Lemma \ref{lem:cm1} and since products of CM functions are again CM (cf. \cite[Cor. 1.6]{rene-book}).\\
%Further if $k(x)=\sum_{j=1}^n w_j e^{-\lambda_j x}$ for some weights $w_j>0$, then
%$$f(x)=\frac{k(x)-k(c x)}{x} = \sum_{j=1}^n w_j \frac{e^{-\lambda_j x}-e^{-\lambda_j x c}}{x} $$
%is CM as weighted sum of CM functions.\\
%Finally observe that every CM function $k(x)$ can be written as integral mixture of functions $e^{-\lambda x}$, $\lambda>0$. Since the set of CM functions is closed under pointwise convergence 
Now let $k$ be an arbitrary CM function, i.e.
$$k(x)=\int_{[0,\infty)} e^{-\lambda x} \rho(d\lambda).$$
Then 
\begin{align*}
 f(x)&= \frac{k(x)-k(c x)}{x} 
%&= \frac{\int_0^\infty e^{-\lambda x} \rho(d\lambda) - \int_0^\infty e^{-\lambda c x} \rho(d\lambda)}{x}\\
=\int_{[0,\infty)} \frac{e^{-\lambda x}- e^{-\lambda c x}}{x} \rho(d\lambda)=\int_{(0,\infty)} \frac{e^{-\lambda x}- e^{-\lambda c x}}{x} \rho(d\lambda)
\end{align*}
 is an integral mixture of CM functions and hence CM.
\end{proof}

Now we can state the proof of Proposition \ref{prop:GGCfactorinBO}.

\begin{proof}[Proof of Proposition \ref{prop:GGCfactorinBO}]
 Assume $\mu\in\T$, then its Laplace exponent is given by 
\begin{equation*}% \label{eq:thorinbernstein}
 \psi_\mu(u)= au + \int_0^\infty (1-e^{-u t}) \frac{k(t)}{t} d t, \quad u\geq 0,
\end{equation*}
for some $a\geq 0$ and a CM function $k$. Hence the Laplace exponent of its $c$-factor $\mu_c$, $c\in (0,1)$, is by \eqref{eq-SD-bernsteindifferenz}
\begin{align*}
 \psi_{\mu_c}(u)&=\psi_\mu(u)-\psi_\mu(cu) = a(1-c)u + \int_0^\infty (1-e^{-u t}) \frac{k(t)-k(c^{-1}t)}{t} dt
\end{align*}
and $\mu_c$ is in Bondesson's class if and only if 
$$f(t)=\frac{k(t)-k(c^{-1}t)}{t}$$
is CM. This holds by Lemma \ref{lem:cm2}.\\
Analogous calculations show that also $\mu_c$, $c>1$, is in Bondesson's class.\\
For the converse assume $\mu\in \SD^+$ with $\mu_c\in \BO$ for all $c\in(0,1)$, i.e.
 $\psi_{\mu_c}(u)=\psi_\mu(u)-\psi_\mu(cu)$ is a CBF for all $c\in(0,1)$. 
This implies that 
$$\psi_X(u):=u\psi'_\mu(u) = u \lim_{c\to 1} \frac{\psi_\mu(u)-\psi_\mu(u-(1-c)u)}{u(1-c)}= \lim_{c\to 1} \frac{\psi_\mu(u)-\psi_\mu(u-(1-c)u)}{(1-c)}  $$
is the limit of CBFs and hence a CBF (\cite[Cor. 7.6]{rene-book}). Similarly, if $\mu_c\in \BO$ for all $c>1$ one obtains $\psi_X(u)$ as limit of CBFs for $c\searrow 1$.\\
Now let $(X_t)_{t\geq 0}$ be the subordinator with Laplace exponent $\psi_X$, then  by \cite[Thm. 4 (ii)]{BLMranges} (setting $\sigma=0$) this is equivalent to
$\mu=\Phi_\xi(\cL(X_1))$ for $\xi_t=t$. Hence by Proposition \ref{prop:bondessontoGGC} $\mu$ is in $\T$.
\end{proof}

\section*{Acknowledgements}
The author thanks Makoto Maejima for an inspiring discussion which led to the given proof of Proposition \ref{prop:bondessontoGGC}. Alexander Lindner is thanked for  showing constant interest in this work. Further thanks go to the referee for her/his positive and instructive report which helped to improve the paper.


\begin{thebibliography}{99}
%
\bibitem{alsmeyeretal}{Alsmeyer, G., Iksanov, A. and
R\"osler, U. (2009) On distributional properties of perpetuities. {\it
J. Theoret. Probab.} {\bf  22}, 666--682.}

\bibitem{BarndorffNielsen-Shephard}{Barndorff--Nielsen, O. E. and Shephard, N. (2001)
Modelling by L\'evy processes for financial econometrics. In: Barndorff-Nielsen, O. E., Mikosch, T., Resnick, S. (eds.): {\it L\'evy Processes: Theory and Applications}. 283--318, Birkh\"auser, Boston.} 

\bibitem{MaejimaBarndorff}{Barndorff-Nielsen, O. E., Maejima, M. and Sato, K. (2006). Some classes of multivariate infinitely divisible distributions admitting stochastic
integral representations. {\it Bernoulli} {\bf 12}, 1--33.}

%\bibitem{behme2011}A. Behme (2011) Distributional properties of
%solutions of $dV_t = V_{t-} dU_t + dL_t$ with L\'evy noise. {\it
%Adv. Appl. Prob.} {\bf 43}, 688--711.


%\bibitem{behmelindnermaller11}A. Behme, A. Lindner and R. Maller
%(2011) Stationary solutions of the stochastic differential equation
%$dV_t = V_{t-} dU_t + dL_t$ with L\'evy noise. {\it Stoch. Proc.
%Appl.} {\bf 121}, 91--108.
\bibitem{BehmeBondesson}{Behme, A. and Bondesson, L. (2015+) A class of scale mixtures of gamma(k)-distributions that are generalized gamma convolutions. Submitted. Preprint available on https://mediatum.ub.tum.de/node?id=1243280.}

\bibitem{BLexpfunc}{Behme, A. and Lindner, A. (2012) On exponential
functionals of L\'evy processes. {\it J. Theor. Probab.} doi:10.1007/s10959-013-0507-y.}

\bibitem{BLMranges}Behme, A., Lindner, A. and Maejima, M. (2014+) Ranges of exponential functionals of L\'evy processes, to appear in {\it S\'eminaire de Probabilit\'es}.


\bibitem{BehmeMaejima}{Behme, A., Maejima, M., Matsui, M. and Sakuma, N.
(2012) Distributions of exponential integrals of independent
increment processes related to generalized gamma convolutions. {\it
Bernoulli}, {\bf 18}, 1172--1187.}


\bibitem{bertoinlindnermaller08}{Bertoin, J., Lindner, A. and Maller, R.
(2008) On continuity properties of the law of integrals of L\'evy
processes. In:
 C. Donati-Martin, M. \'Emery, A. Rouault, C. Stricker (eds.):
 {\it S\'eminaire de Probabilit\'es XLI, Lecture
  Notes in Mathematics} {\bf 1934}, 137--159, Springer, Berlin.}


\bibitem{bertoinyor} {Bertoin, J. and Yor, M. (2005) Exponential
functionals of L\'evy processes. {\it Probab. Surveys} {\bf 2},
191--212.}

\bibitem{bondesson81}{Bondesson, L. (1981) Classes of infinitely divisible distributions
and densities. {\it Z. Wahrscheinlichkeitstheorie verw. Gebiete} {\bf 57}, 39--71.}


\bibitem{CarmonaPetitYor}{Carmona, P., Petit, F. and Yor, M. (1997) On
the distribution and asymptotic results for exponential functionals
of L\'evy processes. In {\it Exponential Functionals and Principal
Values Related to Brownian Motion}, Bibl. Rev. Mat. Iberoamericana,
Rev. Mat. Iberoamericana, Madrid, 73--130. }

\bibitem{ericksonmaller05}{Erickson, K. B. and Maller, R. A. (2005) Generalised
  Ornstein-Uhlenbeck processes and the convergence of L\'evy integrals; in
  M. Emery, M. Ledoux, M. Yor (eds.): {\it S\'eminaire de Probabilit\'es XXXVIII, Lecture
  Notes in Mathematics} {\bf 1857}, 70--94, Springer, Berlin.}


\bibitem{GjessingPaulsen} Gjessing, H. K. and Paulsen, J. (1997) Present value distributions with applications to ruin theory and stochastic equations. {\it Stoch. Proc. Appl.} {\bf 71}, 123--144.

\bibitem{Hutchinson}{Hutchinson, J. E. (1981) Fractals and self-similarity. {\it Indiana Univ. Math. J.} {\bf 30}, 713--747.}

\bibitem{KLM:2004}{Kl{\"u}ppelberg, C., Lindner, A. and Maller, R. (2004) A continuous-time {GARCH} process driven by a {L}\'evy process: stationarity and second-order behaviour. {\it J. Appl. Probab.} {\bf 41}, 601--622.}


\bibitem{kuznetsovetal}Kuznetsov, A., Pardo, J. C. and Savov, M. (2012) Distributional properties
of exponential functionals of L\'evy processes. {\it Electron. J. Probab.} {\bf 17}, 1--35.



\bibitem{lindnermaller05} {Lindner, A. and Maller, R. (2005) L\'evy integrals
  and the stationarity of generalised Ornstein-Uhlenbeck processes.
  {\it Stoch. Process. Appl.} {\bf 115}, 1701--1722.}

\bibitem{lindnersato09} {Lindner, A.  and Sato, K. (2009) Continuity
    pro\-per\-ties and infinite divisibility of stationary distributions of
   some generalised Ornstein--Uhlenbeck processes. {\it Ann. Probab.} {\bf 37},
250--274.}


\bibitem{maejimalevymatters}{Maejima, M. (2015+) Classes of infinitely divisible distributions and examples. In: Barndorff-Nielsen, O.E., Bertoin, J., Jacod, J. and Kl\"uppelberg, C. (eds.): {\it L\'evy Matters V}, Springer, Berlin.}


\bibitem{maischenkscherer}{Mai, J.-F., Schenk, S. and Scherer, M. (2014+) Two novel characteri\-zations of self-decomposability on the half-line. Submitted. Preprint available on http://mediatum.ub.tum.de/node?id=1200670.}

\bibitem{PardoPatieSavov}{Pardo, J. C., Patie, P. and Savov, M. (2012). A Wiener-Hopf type factorization of the exponential functional of L\'evy processes. {\it J. London Math. Soc.} {\bf 86}, 930--956.}

\bibitem{PardoRiveroSchaik} Pardo, J. C., Rivero, V. and van Schaik, K. (2013+) On the density of exponential functionals of L\'evy processes. {\it Bernoulli}, to appear.


\bibitem{sato}{Sato, K. (1999) {\it L\'evy Processes and Infinitely Divisible
  Distributions}. Cambridge University Press, Cambridge.}

\bibitem{satolevymatter}{Sato, K. (2010) Fractional integrals and extensions of selfdecomposability. In: Barndorff-Nielsen, O.E., Bertoin, J., Jacod, J. and Kl\"uppelberg, C. (eds.): {\it L\'evy Matters I}, 1--91, Springer, Berlin. }


\bibitem{rene-book}{Schilling, R. L., Song, R. and Vondracek, Z. (2012) {\it Bernstein Functions. {T}heory and {A}pplications.}
2nd edition. De {G}ruyter {S}tudies in {M}athematics 37. {B}erlin: {W}alter de {G}ruyter. }


\bibitem{steutelvanharn}{Steutel, F. W. and van Harn, K. (2003)
{\it Infinite Divisibility of Probability Distributions on the Real
Line}. Marcel Dekker Inc, New York. }

%\bibitem{wolfe}
% Wolfe, S. J. (1982) Continuity properties of decomposable probability measure on Euclidean spaces.
%{\it J. Multivariate Anal.}
%{\bf 13},
%534--538.

\end{thebibliography}
\end{document}